\documentclass{amsart}

\usepackage{amssymb}
\usepackage{amsfonts}
\usepackage{amsthm}
\usepackage{hyperref}
\usepackage{epsfig,amsmath}
\usepackage{graphicx}
\usepackage{caption}
\usepackage{subcaption}
\usepackage{wrapfig}
\usepackage{xcolor}

\usepackage[margin=1.00 in]{geometry}

\makeatletter
\def\ps@pprintTitle{%
 \let\@oddhead\@empty
 \let\@evenhead\@empty
 \def\@oddfoot{}%
 \let\@evenfoot\@oddfoot}
\makeatother

\newtheorem{thm}{Theorem}
\newtheorem{maintheorem}{Theorem}

\newtheorem{prop}[thm]{Proposition}
\newtheorem{lem}[thm]{Lemma}
\newtheorem{coro}[thm]{Corollary}

\newtheorem{conj}[thm]{Conjecture}

\newtheorem{ques}{Question}

\newtheorem{claim}[thm]{Claim}
\theoremstyle{definition}
\newtheorem{constr}[thm]{Construction}

\newcommand{\para}[1]{\medskip\noindent\textbf{#1.}}

\newcommand{\p}[2]{\frac{\partial #1}{\partial #2}}
\newcommand{\TT}{\mathcal{T}}

\newcommand{\Ext}{\mathsf{Ext}}
\numberwithin{thm}{section}

\begin{document}

\title{Generalized bipyramids and hyperbolic volumes of alternating $k$-uniform tiling links}

\begin{abstract}
We present explicit geometric decompositions of the hyperbolic complements of \emph{alternating $k$-uniform tiling links}, which are alternating links whose projection graphs are $k$-uniform tilings of $S^2$, $\mathbb{E}^2$, or $\mathbb{H}^2$. A consequence of this decomposition is that the volumes of spherical alternating $k$-uniform tiling links are precisely twice the maximal volumes of the ideal Archimedean solids of the same combinatorial description, and the hyperbolic structures for the hyperbolic alternating tiling links come from the equilateral realization of the $k$-uniform tiling on $\mathbb{H}^2$. In the case of hyperbolic tiling links, we are led to consider links embedded in thickened surfaces $S_g \times I$ with genus $g \ge 2$ and totally geodesic boundaries. We generalize the bipyramid construction of Adams to truncated bipyramids and use them to prove that the set of possible volume densities for all hyperbolic links in $S_g \times I$, ranging over all $g \ge 2$, is a dense subset of the interval $[0, 2v_\text{oct}]$, where $v_\text{oct} \approx 3.66386$ is the volume of the  ideal regular octahedron.
\end{abstract}

\author[C. Adams]{Colin Adams}
\address[Colin Adams]{Williams College}
\email{colin.c.adams@williams.edu}

\author[A. Calderon]{Aaron Calderon}
\address[Aaron Calderon]{Yale University}
\email{aaron.calderon@yale.edu}

\author[N. Mayer]{Nathaniel Mayer}
\address[Nathaniel Mayer]{University of Chicago}
\email{nmayer26@gmail.com}

\date{\today}

\maketitle

\section{Introduction}
For many links $L$ embedded in a 3-manifold $M$, the complement $M \setminus L$ admits a unique hyperbolic structure. In such cases we say that the link is hyperbolic, and denote the hyperbolic volume of the complement by $\text{vol}(L)$, leaving the identity of $M$ implicit. Computer programs like  SnapPy allow for easy numerical computation of these volumes, but exact theoretical computations are rare.

In \cite{Champanerkar}, Champanerkar, Kofman, and Purcell present an explicit geometric decomposition of the complement of the infinite square weave, the infinite alternating link whose projection graph is the square lattice $\mathcal{W}$ (see Figure \ref{inftweave}) into regular ideal octahedra, one for each square face of the projection. The volume of the infinite link complement is infinite, but it is natural to study instead the \emph{volume density}, defined as $\mathcal{D}(L) = \text{vol}(L)/c(L)$ where $c(L)$ is the crossing number. For infinite links with symmetry like $\mathcal{W}$, we can make this well defined by taking the volume density of a single fundamental domain. Since faces of $\mathcal{W}$ are in bijective correspondence with crossings, the decomposition of \cite{Champanerkar} shows the volume density to be $\mathcal{D}(\mathcal{W}) = v_\text{oct} \approx 3.66386$, the volume of the regular ideal octahedron and the largest possible volume density for any link in $S^3$. 

\begin{figure}[htbp]
\centering
\includegraphics[scale=0.5]{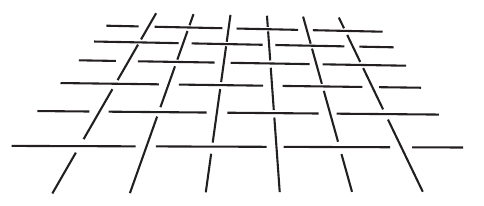}
\caption{The infinite square weave link $\mathcal{W}$, a tiling link derived from the square tiling of the Euclidean plane.}
\label{inftweave}
\end{figure}

In this paper we generalize their example and give explicit geometric decompositions and volume density computations for  \emph{infinite  alternating $k$-uniform tiling links}, alternating links whose projection graphs are edge-to edge tilings by regular polygons of the sphere (embedded in $S^3$), the Euclidean plane (in $\mathbb{R}^3$), or the hyperbolic plane (in $\mathbb{H}^2 \times I$) such that their symmetry groups have compact fundamental domain. (The terminology $k$-uniform refers to the fact there are $k < \infty$ transitivity classes of vertices). In the cases of Euclidean tilings and hyperbolic tilings, we avoid working directly with infinite links by taking a quotient of the infinite link complement by a surface subgroup of the symmetry group. In the case of a Euclidean tiling, we quotient by a $\mathbb{Z} \times \mathbb{Z}$ subgroup and obtain a finite link complement in a thickened torus $T \times (0,1)$. Independently, in \cite{Champanerkar6}, the authors also determine the hyperbolic structures for these alternating tiling links. 

 In the case of a hyperbolic tiling,  we consider thickened higher genus surfaces, denoted $S_g \times I$ where $g$ is the genus. We restrict to orientable surfaces until the very end.
 When the specific genus is unimportant, we use $S \times I$ to denote any thickened surface of genus $g \ge 2$.
 
 The crossing number of a link in $S \times I$ is defined as the minimal number of crossings across all possible projections onto $S \times \{0\}$. It follows from  \cite{AFLT02} and\cite{Fleming} that for an alternating link in $S \times I$, $c(L)$ is realized only for a reduced alternating projection of $L$ onto $S \times \{0\}$. We assume throughout that the two boundary components of $S \times I$ in the hyperbolic structure on $(S \times I)  \setminus L$ are totally geodesic, and hence the metric is uniquely determined (see Theorem \ref{uniquetotgeostructure}).  

Our main theorem relates the geometry of the link complement $(S \times I) \setminus L$ to the geometry and combinatorics of the tiling.

\begin{maintheorem}[c.f. Theorem \ref{computingvolumes}]\label{mainthm}
The hyperbolic structure on a alternating $k$-uniform  tiling link is realized by a polyhedral decomposition whose dihedral angles are specified by the angles of (the equilateral realization of) the corresponding tiling.
\end{maintheorem}
  
This analysis of alternating $k$-uniform tiling links also yields results about general hyperbolic link complements in $T \times (0,1)$ and $S \times I$.

Generalizing a result of D. Thurston \cite{Thurston2}, we show in Corollary \ref{2voctdensity} that volume density for links in $S \times I$ is bounded above by $2v_\text{oct}$.  
This turns out to be the only restriction on the spectrum of volume densities, for we show that

\begin{maintheorem}[c.f. Theorem \ref{2voctdensity}]\label{mainthm:density} 
The set of volume densities of links in $S_g \times I$ across all genera $g \ge 2$, assuming totally geodesic boundary, is a dense subset of the interval $[0, 2v_\text{oct}]$.
\end{maintheorem}

As a corollary to the proof, we have that volume densities for links in the thickened torus $T \times (0,1)$  are  dense in the interval $[0,v_\text{oct}]$. This extends previous results that volume densities are dense in $[0,v_\text{oct}]$ for knots in $S^3$ from \cite{Burton2015} and \cite{Adams2017}. 


\para{Outline}
In Section \ref{sec:tilinglinks}, we introduce the main objects of our study and explain how certain tilings of $\mathbb{S}^2$, $\mathbb{E}^2$, or $\mathbb{H}^2$ give rise to alternating links in thickened surfaces. After introducing this class of links, we show in Section \ref{subsec:hypstr} that their complements admit a unique hyperbolic structure (with totally geodesic boundary when applicable). In Section \ref{subsec:decomp} we give topological decompositions of these link complements, generalizing both the octahedral decomposition of \cite{Thurston2} and the bipyramidal decomposition of \cite{Adams2015}.

In Section \ref{sec:polyhedra}, we present the relevant background on generalized hyperbolic polyhedra, and use this opportunity to investigate maximal volume tetrahedra and bipyramids. This yields bounds on the volumes of link complements in thickened surfaces in terms of the combinatorics of their link diagrams.

After determining explicit hyperbolic structures on these building blocks, in Section \ref{sec:volumes} we prove that the symmetry of the tiling forces symmetry of the corresponding link complement.
In particular, this symmetry allows us to explicitly determine the dihedral angles of the polyhedra in the decompositions of Section \ref{sec:polyhedra}, leading to our exact computation of volumes and the proof of Theorem \ref{mainthm}. 

Finally, in Section \ref{sec:density} we prove Theorem \ref{mainthm:density}, which shows that the relationship between volume and crossing number is no stronger than the bound established in Section \ref{sec:polyhedra}. The main tools here are covering spaces and a bound on the volumes of hyperbolic Dehn fillings due to Futer, Kalfagianni, and Purcell \cite{Futer08}.

\para{Acknowledgments}
This research began at the SMALL REU program in the summer of 2015, and we are very grateful for our teammates Xinyi Jiang, Alexander Kastner, Greg Kehne and Mia Smith for their hard work and helpful conversations along the way. We are also grateful to Yair Minsky for helpful comments, as well as Abhijit Champanerkar, Ilya Kofman, and Jessica Purcell for their useful input.
This paper owes much to anonymous referees who gave us particularly detailed and helpful feedback on earlier versions of the paper.
The authors were supported in part by NSF grant DMS-1659037, the Williams College SMALL REU program, and NSF graduate fellowship DGE-1122492.

\section{Alternating k-Uniform Tiling Links}\label{sec:tilinglinks}

In this section, we introduce {\em alternating $k$--uniform tiling links}, the main objects of our study in this paper (Construction \ref{constr:tilinglink}). These correspond to links in thickened surfaces, and by a criterion of \cite{Aetal18} (Theorem \ref{thm:alternatinghyp}), correspond to links with hyperbolic exteriors. As a consequence of Mostow--Prasad rigidity, the hyperbolic structure on the link exterior is unique (Theorem \ref{uniquetotgeostructure}), so long as the boundary surfaces are assumed to be totally geodesic for genus at least 2.

After defining tiling links, we introduce certain (topological) polyhedral decompositions of their exteriors (Lemmas \ref{lem:octdecomp} and \ref{lem:bipdecomp}). These decompositions generalize D. Thurston's octahedral decomposition \cite{Thurston2} as well as Adams's bipyramid construction \cite{Adams2015}. The rest of the paper is then devoted to determining the exact geometric structure corresponding to these decompositions.

\subsection{Links from tilings}

The goal of this section is to provide a method to produce, given a suitable tiling, an alternating link in a thickened surface (Construction \ref{constr:tilinglink}). The tiling becomes an alternating link diagram, with vertices corresponding to crossings and edges to arcs between them.

Recall that a tiling $\mathcal{T}$ of $\mathbb{X} \in \{\mathbb{S}^2, \mathbb{E}^2, \mathbb{H}^2 \}$ is said to be {\it $k$-uniform} if it consists of regular polygonal tiles meeting edge-to-edge such that there are $k$ transitivity classes of vertices and a compact fundamental domain for the symmetry group of the tiling. Note that the assumption of compact fundamental domain forces each tile itself to be compact. By the assumption of $k$--uniformity, $\mathcal{T}$ induces a tiling of a compact surface.

\begin{lem}\label{lem:tilingonsurf}
Let $\mathcal{T}$ be a $k$--uniform tiling of
$\mathbb{X} \in \{\mathbb{S}^2, \mathbb{E}^2, \mathbb{H}^2 \}$.
Then there exists a compact quotient $S = \mathbb{X}/\Gamma$ of $\mathbb{X}$ such that $\mathcal{T}$ descends to a tiling of $S$.
\end{lem}
\begin{proof}
When $\mathbb{X} = \mathbb{S}^2$, the statement is vacuously true with $\Gamma = \{1\}$.

Otherwise, let $G_\mathcal{T} \le \text{Isom}_+(\mathbb{X})$ be the orientation--preserving symmetry group of $G_\mathcal{T}$. By the assumption of $k$--uniformity, the action of $G_\mathcal{T}$ has compact fundamental domain.

Therefore, by Selberg's lemma, $G_\mathcal{T}$ has a torsion--free normal subgroup $\Gamma$ of finite index. Since $\Gamma$ is torsion free and acts cocompactly on $\mathbb{X}$, the quotient $S= \mathbb{X}/\Gamma$ is a compact surface of genus $g \ge 1$ (depending on whether $\mathbb{X} = \mathbb{E}^2$ or $\mathbb{H}^2$).
\end{proof}

When $\mathcal{T}$ is 4--regular (i.e each vertex meets four tiles), any choice of over-- or under--crossing at each vertex will result in a link diagram on $S$. In order to get link diagrams from 3--regular tilings, we will turn some edges into {\em bigon tiles} to make each vertex 4--regular. Eventually, bigon tiles become bigons in the link diagram (see, e.g., Figure \ref{Fig:euclidtilingexample}).

\begin{lem}\label{lem:make4reg}
Any $3$--regular tiling of a compact surface can be made into a 4--regular tiling with bigon tiles.
\end{lem}
\begin{proof}
View the tiling $\mathcal{T}$ as a graph embedded on the surface. Since the graph is 3--regular and every edge bounds a tile (and hence belongs to a cycle), Petersen's Theorem implies that there exists a perfect matching (see, \emph{e.g.} \cite{Petersen}, \cite{Brahana}). That is to say, there is a subcollection $E$ of the edges of $\mathcal{T}$ such that each vertex is an endpoint of exactly one edge of $E$.

Replacing each edge in the subcollection with a bigon yields a 4--regular tiling of $S$.
\end{proof}

We note that if a tiling of $S$ has a mixture of 3--valent and 4--valent vertices such that there is a collection of edges that pair the 3--valent vertices, then we can replace those edges with bigons to obtain a 4--regular tiling with bigon tiles.

As above, any choice of over-- or under--crossing at each vertex of a 4--regular tiling with bigon tiles results in a link diagram on a surface. To ensure that the diagram is alternating, however, we may need to pass to a finite cover of $S$. This is the content of the next lemma.

We say that a tiling with bigon tiles can be {\em checkerboard shaded} if its tiles (including the bigons) can be colored black and white so that no two adjacent tiles have the same color.

\begin{lem}\label{lem:checkers}
Let $\mathcal{T}$ be a 4--regular tiling (possibly with bigon tiles) on a surface $S$. Then there exists some finite cover $S'$ of $S$ so that the lifted tiling $\mathcal{T}'$ of $S'$ can be checkerboard shaded.
\end{lem}
\begin{proof}
In order to determine the cover, we will shade $\mathcal{T}$ and then lift to a cover to make our shading a checkerboard.

For each tile $D$ of $\mathcal{T}$, choose a point $x_D$ living in the interior and pick a basepoint $x_0$ from among the $x_D$.
For each $D$, choose a path $\gamma_D$ from $x_0$ to $x_D$ which avoids the vertices of $\mathcal{T}$.

If  $\gamma_D$ intersects $\mathcal{T}$ an even number of times, color $D$ black, and if an odd number of times, color $D$ white.
As $\mathcal{T}$ is 4--regular, the parity of the number of intersections of $\gamma_D$ with the edges of $\mathcal{T}$ is homotopy invariant, allowing homotopies (rel $x_0$ and $x_D$) which pass over a vertex.

The construction above yields a shading of $\mathcal{T}$, but it may not be checkerboard. If it is not, then there are two adjacent tiles $D$ and $D'$ which are colored the same way.

In that case, let $\delta$ denote a path from $x_D$ to $x_{D'}$ which meets $\mathcal{T}$ only once. Then $\gamma_D \circ \delta \circ \gamma_{D'}^{-1}$ is a loop from $x_0$ to itself which by construction meets $\mathcal{T}$ an odd number of times.

More generally, the tiling $\mathcal{T}$ defines a homomorphism
\[i:\pi_1(S, x_0) \rightarrow \mathbb{Z}/2\mathbb{Z}\]
which records whether a loop meets $\mathcal{T}$ an odd or even number of times.

Taking $S'$ to be the cover corresponding to the kernel of $i$ and lifting $\mathcal{T}$, we see that any loop in $\pi_1(S', x_0')$ meets $\mathcal{T}'$ an even number of times, and hence the shading of $\mathcal{T}$ lifts to a checkerboard shading of $\mathcal{T}'$.
\end{proof}

With the above preliminaries, we can now show how certain tilings of $\mathbb{X}$ give rise to links in thickened surfaces.

\begin{constr}[Tiling links]\label{constr:tilinglink}
Let $\mathcal{T}$ be a $k$--uniform tile that is either 3-- or 4--regular, or has a mix of 3-- and 4--valent vertices such that there exists a collection of edges which pairs the 3--valent vertices.

Let $\mathbb{X}/\Gamma$ be the compact surface tiled by $\mathcal{T}$, as guaranteed by Lemma \ref{lem:tilingonsurf}. If necessary, double edges as in Lemma \ref{lem:make4reg} to turn $\mathcal{T}$ into a 4--regular tiling, possibly with bigon tiles. Finally, lift $\mathcal{T}$ to a tiling of the finite cover $S$ of $\mathbb{X} / \Gamma$ produced by Lemma \ref{lem:checkers}.

We may now resolve the vertices of $\mathcal{T}$ into over-- and under--crossings so that the resulting link diagram is alternating. For each white tile of $\mathcal{T}$, we stipulate that the strands on the boundary go from under a crossing to over a crossing as we travel clockwise around the boundary. For the black tiles, the strands on the boundary go from under a crossing to over a crossing as we travel counterclockwise around the boundary. This yields a consistent way to choose crossings so the resultant link diagram is alternating.

We can now treat the link as living in $M = S_g \times I$. In the case of a spherical tiling, we cap off the two spherical boundaries with balls to obtain a link $L$ in $M' = S^3$.
In the case of a Euclidean tiling, the link lives in $M' = T \times (0,1)$. In the case of a hyperbolic tiling, $S_g$ has genus at least 2 and the link lives in $M' = M = S_g \times I$. 
 
 We call such a link $L$ in $M'$ an {\it alternating $k$-uniform tiling link}. We can then lift the choice of crossings back to the link in $\mathbb{E}^2 \times (0,1)$ or $\mathbb{H}^2 \times I$ to obtain the corresponding {\it infinite alternating $k$-uniform tiling link.}
 \end{constr}
 
In the sequel, we often suppress the alternating and $k$--uniform descriptors and call any link resulting from the above construction a {\em tiling link}. We use $\Ext(L)$ to denote the complement of $L$ in $M'$.

 Note that there are multiple ways to resolve a $k$-uniform tiling into an infinite alternating $k$-uniform tiling links due to our choice of how to put in a crossing at a vertex, to where we choose to add bigons and to the size of the fundamental domain we choose. For $g >0$, by taking larger and larger fundamental domains, the number of possible links grows exponentially. For example, the 6.6.6 tiling has an infinite number of ways we may add bigons to make it a $4$--regular graph, all of which respect some $\mathbb{Z}^2$-subgroup of the symmetry group of the tiling. (The notation $p.q.r. \dots.s$ means that around every vertex, we have in consecutive order a $p$-gon, a $q$-gon, an $r$-gon,$\dots$ and an $s$-gon). Our results hold independently of the choice of such bigons.

 The alternating $k$-uniform tiling links derived from spherical tilings correspond precisely to Archimedean polyhedra. Although we include that case here, note that in \cite {AR}, the authors classified exactly the alternating hyperbolic links corresponding to the 3--regular and 4--regular Archimedean solids.

\begin{figure}[htbp]
\centering
\begin{subfigure}[b]{.48 \textwidth}
\centering
\includegraphics[scale=1]{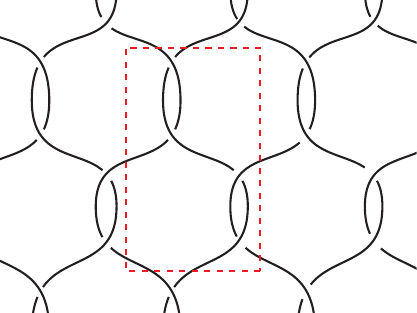}
\caption{A tiling link derived from the 6.6.6 tiling of $\mathbb{E}^2$.}
\label{Fig:euclidtilingexample}
\end{subfigure}
\begin{subfigure}[b]{.48 \textwidth}
\centering
\includegraphics[scale=.2]{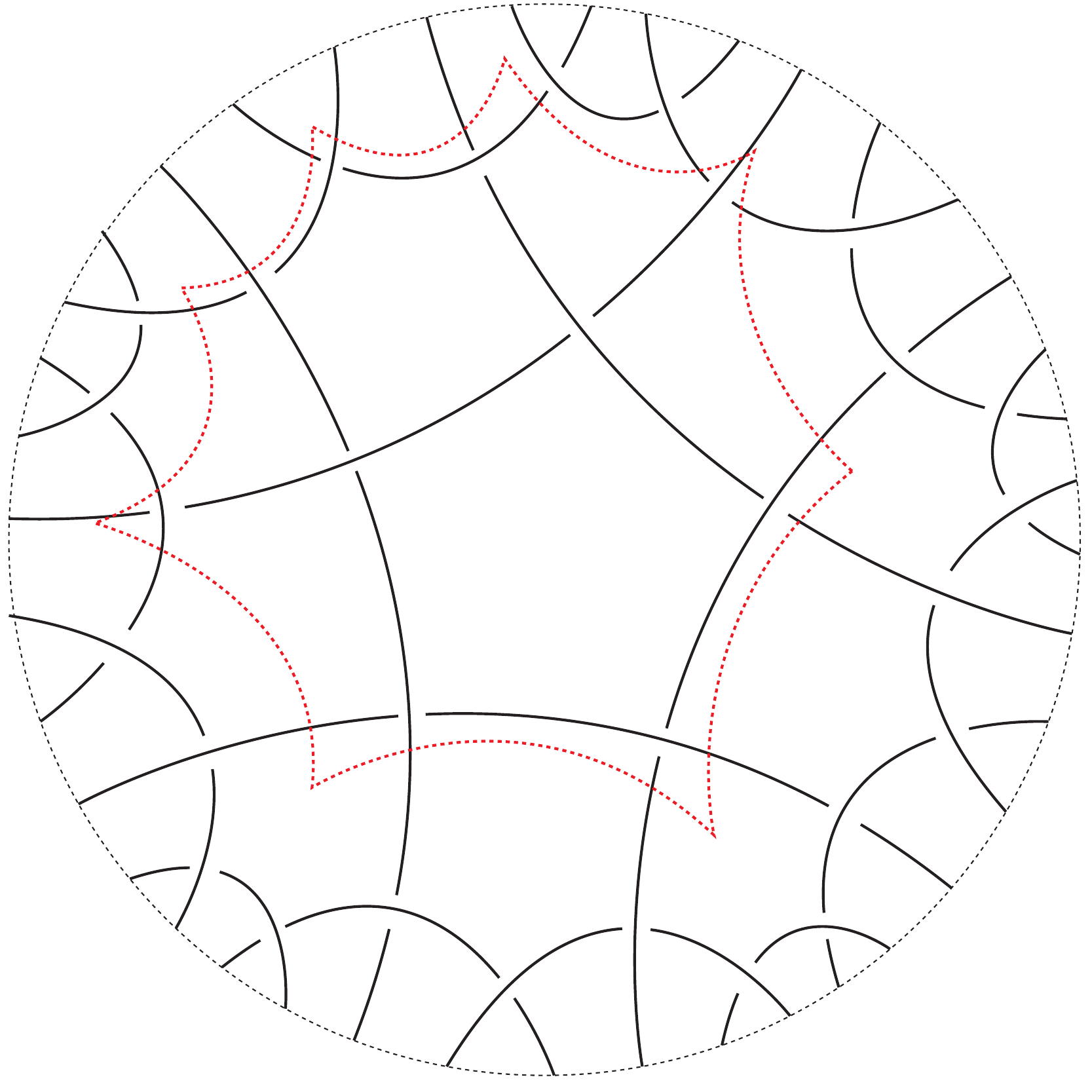}
\caption{A tiling link derived from the 5.5.5.5 tiling of $\mathbb{H}^2$.}
\label{Fig:hyptilingexample}
\end{subfigure}
\caption{Tiling links shown with fundamental domains for appropriate surface groups.}
\label{Fig:tilingexamples}
\end{figure}

\subsection{Hyperbolic structures on tiling link exteriors}\label{subsec:hypstr}

In this subsection we show that when $g \ge 2$, the complement of an alternating $k$-uniform tiling link in $S_g \times I$ can be given a  unique hyperbolic metric such that both $S_g \times \{0\}$ and $S_g \times \{1\}$ are totally geodesic. We begin with a well-known fact, but include a proof for completeness.

\begin{thm}\label{uniquetotgeostructure} A finite volume anannular hyperbolic  3-manifold with boundary has exactly one finite-volume complete hyperbolic structure with totally geodesic boundary.
\end{thm}

\begin{proof}Let $M$ be an anannular  finite volume hyperbolic 3-manifold with boundary components of genus greater than 1. By doubling the manifold along its boundaries, we obtain a manifold $DM$ with no boundary and a finite number of toroidal cusps. Because $M$ is anannular and hyperbolic, and therefore also atoroidal, irreducible and boundary-irreducible, $DM$ must also be anannular, atoroidal, irreducible and boundary-irreducible. By Thurston's geometrization theorem for Haken manifolds (see \cite{Thurston1982}), $DM$ admits a complete finite volume hyperbolic structure, and the Mostow--Prasad Rigidity Theorem implies this structure is unique.  The fact $DM$ has an orientation-reversing involution with fixed point set the boundary of $M$ implies by Mostow--Prasad rigidity that the boundary of $M$ will be realized as totally geodesic surfaces in $DM$ (see for instance Lemma 1 of \cite {AR92}.) 
\end{proof}

 For link complements in $S_g \times I$, with genus $g > 1$, we concern ourselves exclusively with the case where $S_g \times \{0\}$ and $S_g \times \{1\}$ are totally geodesic. This satisfies the hypothesis of Theorem \ref{uniquetotgeostructure} to ensure that a complete hyperbolic structure of finite volume is unique if it exists.  Throughout this section, we often suppress the genus and refer to any thickened surface of genus $g \ge 2$ by $S \times I$. A link projection on a surface is {\it fully alternating} if it is alternating and all complementary regions are disks. (This is sometimes called a cellular embedding in the literature).
 We utilize two theorems from \cite{Aetal18}.

\begin{thm}\cite{Aetal18}
\label{prime}A reduced fully alternating link $L'$ on a surface  $S$  is prime if and only if there does not exist a disk $D$ on $S$ such that $\partial D$ intersects $L'$ twice transversely and there are crossings in $D\cap L'$.
\end{thm}

\noindent When there are no such disks, we say $(S,L')$ is \textit{obviously prime}. Thus the theorem says that a reduced fully alternating link in a thickened surface is prime if and only if it is obviously prime. 

\begin{thm}\cite{Aetal18}
\label{tg}A prime fully alternating link in a thickened orientable surface of genus at least one is hyperbolic, and if $g \geq 2$, the boundary surfaces can be taken to be totally geodesic. 
\end{thm}

We can then prove the following theorem.

 \begin{thm} For $g \ge 2$, the complement of an alternating $k$-uniform tiling link in $S_g \times I$ has a hyperbolic metric with totally geodesic boundaries.
 \label{thm:alternatinghyp}
 \end{thm}
 
 \begin{proof}
  By construction, an alternating $k$-uniform tiling link in $S_g \times I$ has a connected projection that is fully alternating, meaning that the complementary regions are all disks, and that the link alternates its crossings as we travel along any and all components.
  
  Furthermore, there can be no circle bounding a disk on the surface such that it intersects the link projection twice and such that there are crossings inside the disk. If there were such a disk, it would correspond to such a disk in the tiling of $\mathbb{H}^2$, which is impossible for a tiling consisting of regular polygons. Thus, we have an obviously prime reduced fully alternating link projection in $S_g$.
  
  Theorem \ref{prime} implies the link is prime and Theorem \ref{tg} then implies the link complement is hyperbolic (with totally geodesic boundary if $g \geq 2$).
 \end{proof}
 
 Note that in \cite{Aetal18}, a hyperbolic 3--manifold $M$ such that all boundaries of genus at least 2 are totally geodesic is called a {\it tg-hyperbolic 3--manifold}.

\subsection{Polyhedral decompositions of tiling link complements}\label{subsec:decomp}

In this section, we recall some topological decompositions of (classical) link complements and discuss their generalizations to tiling links. In Section \ref{sec:volumes}, we realize these decompositions geometrically to determine the hyperbolic structure on the exterior.

In \cite{Thurston2}, D. Thurston described a decomposition of the complement of a link in $S^3$ into octahedra, using one at each crossing, as in Figure \ref{octahedratobipyramids}(a). These octahedra have two ideal vertices located on the cusps at that crossing and four finite vertices which are identified in pairs to two finite points. The two points are thought of as being far above and below the projection sphere for the link, and we denote them $U$ and $D$ (for up and down). Any such octahedron has volume less than $v_\text{oct}$, so this gives an upper bound on the volume of the link: 
\[\text{vol}(L) \leq v_\text{oct}c(L).\]
There exist links such that $\text{vol}(L) /c(L)$ asymptotically approaches $v_{oct}$. (cf. \cite{Champanerkar}, \cite{Champanerkar2}).

In \cite{Adams2015}, the octahedral decomposition was rearranged into face-centered bipyramids. Thurston's octahedra are cut open along the core vertical line connecting the ideal vertices, yielding four tetrahedra as in Figure \ref{octahedratobipyramids}(b). The tetrahedra each have two ideal and two finite vertices, identified with $U$ and $D$. The edge connecting the finite vertices passes through the center of one of the four faces adjacent to the crossing for that octahedron, as shown in Figure \ref{octahedratobipyramids}. This edge is shared by one tetrahedron from each crossing bordering that face, which glue together to form a bipyramid. The apexes of the bipyramids are the finite vertices $U$ and $D$, while the vertices around the central polygon are ideal. Thus we can think of the link complement as decomposing into bipyramids, one per face of the link projection, each such $n$-bipyramid corresponding to a face of $n$ edges.

Any such $n$-bipyramid has volume less than the maximal volume ideal $n$-bipyramid, shown in \cite{Adams2015} to be regular, with volume bounded above by and asymptotically approaching $2\pi\log(n/2)$. These bipyramids are denoted $B_n^\text{ideal}$. The construction puts an upper bound on volume
\[\text{vol}(L) < \sum_i \text{vol}(B_{n_i}^\text{ideal})\]
where $n_i$ denotes the number of edges in the $i$-th face of the projection of $L$.
For link projections whose faces have many edges, this gives a dramatically better upper bound on volume than $v_{oct} c(L)$, which is the bound generated directly from Thurston's octahedra.

\begin{figure}[htbp]
\centering
\includegraphics[width = 0.4\textwidth]{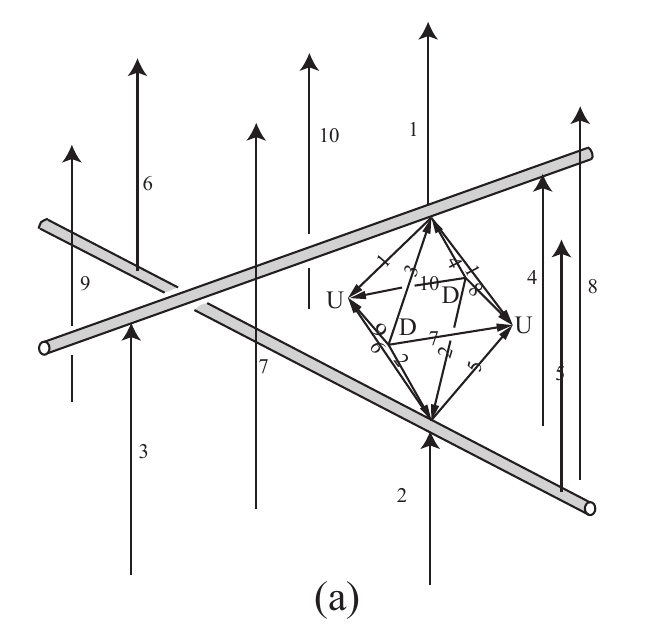}
\qquad
\includegraphics[width = 0.4\textwidth]{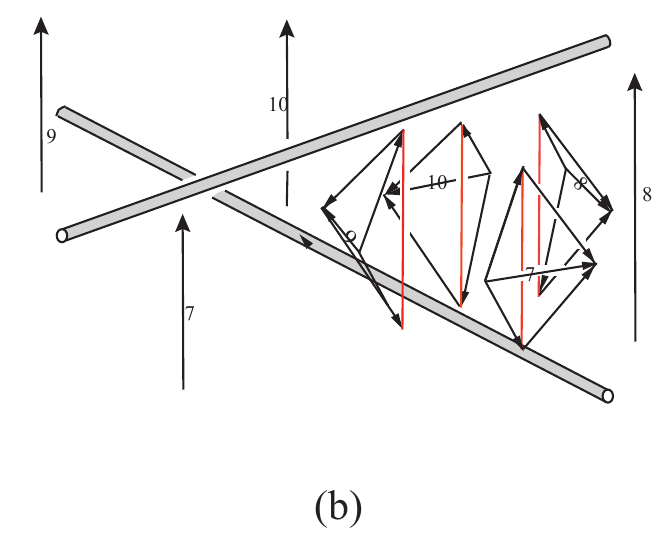}
\caption{Cutting up octahedra to reassemble into face-centered bipyramids.}
\label{octahedratobipyramids}
\end{figure}

We now generalize both the octahedral and bipyramidal decompositions to links in thickened surfaces. Once we have analyzed the maximal volume generalized octahedra and bipyramids, these decompositions will give us analogous bounds on volumes (Corollaries \ref{cor:octvolbd} and \ref{cor:bipvolbd}, respectively).

\begin{lem}\label{lem:octdecomp}
Let $L$ be a link in either $M = T^2 \times (0,1)$ or $M = S_g \times I$ and suppose that $L$ has crossing number $c(L)$. Then $\Ext(L)$ can be (topologically) decomposed into $c(L)$ generalized octahedra, each of which has two vertices on $L$ and all of its other vertices either ideal (if $M = T^2 \times (0,1)$) or truncated ($M = S_g \times I$).
\end{lem}

\begin{proof}
Let $S$ be either $T^2$ or $S_g$, as appropriate.
Recall that the {\em suspension} $\hat{S}$ of a topological space $S$ is the quotient
\[ \hat{S} = (S \times [0,1]) / (S \times \{0\}, S \times \{1\}).\] 
We label the point of $\hat{S}$ corresponding to $S \times \{0\}$ by $D$, and the point corresponding to $S \times \{1\}$ by $U$.

If $M = T^2 \times (0,1)$, we realize $M$ as $S \times (0,1)$ inside of $\hat{S}$, while if $M = S_g \times I$ then realize it as $S_g \times [1/3,\, 2/3]$. This realization extends to an embedding of the link $L$, and therefore we have an embedding $\Ext(L)$ into $\hat{S} \setminus L$. 

We now give an octahedral decomposition of $\hat{S} \setminus L$ and intersect it with $\Ext(L)$. To decompose $\hat{S} \setminus L$ into octahedra, we simply parallel Thurston's proof. Placing an octahedron at each crossing, we may identify the non-ideal vertices to either $U$ or $D$, and identify the faces of the octahedra as in \cite{Thurston2}, yielding a decomposition of $\hat{S} \setminus L$ in which all of the non-link vertices are identified with either $U$ or $D$.

We observe that in the classical setting of a knot diagram on $S^2$, this returns Thurston's original octahedral decomposition of a knot in $S^3 = \widehat{S^2}$.

Now take the intersection of $\Ext(L)$ with this decomposition. If the base surface has genus one, then the intersection of each octahedron with $\Ext(L)$ will be an ideal octahedron (since we excise $U$ and $D$ to get $\Ext(L)$), while if the base surface is of higher genus then we excise neighborhoods of $U$ and $D$, giving truncated octahedra.
\end{proof}
 
 \begin{lem} \label{lem:bipdecomp} Let $L$ be a link in  $M = S \times I$.  Assume  $M' \setminus L$ is hyperbolic. Let $\{F_i\}$ be the collection of $m$ non-bigon faces complementary to the projection of $L$ to $S$ and let $n_i$ be the number of edges in the $i$th face. Then topologically, $\Ext(L)$ can be decomposed into $m$ face-centered bipyramids, each corresponding to a unique face $F_i$ with $n_i$ equatorial edges. When $g = 0$, the apexes of the bipyramids are finite. When $g= 1$, the apexes are ideal. When $g \geq 2$, the apexes are truncated, to yield the two boundary surfaces $S \times \{0\}$ and $S \times \{1\}$.
\end{lem}

\begin{proof} We first consider the spherical case, when $g = 0$. In \cite{Menasco} and implicitly in \cite{ThurstonNotes}, a method is given to decompose a link complement in $S^3$ into two combinatorially equivalent ideal polyhedra, with the gluings on the corresponding faces provided to reproduce the link complement. The faces are complementary to the projection of the link, except that bigons are collapsed to edges. Note that the faces must rotate appropriately before each gluing to obtain the link complement. See \cite {Menasco} for details.

Choosing a single interior point in each polyhedron and coning it to the boundary allows us to then decompose each of the two polyhedra into a collection of pyramids, one corresponding to each face on the projection sphere. Then, gluing together the base of one pyramid from polyhedron 1 with the pyramid with base sharing the corresponding face from pyramid 2 yields the collection of bipyramids that gives the requisite decomposition.

In the case of $g = 1$, we can use the same decomposition of the complement, but instead of two polyhedra, we use two copies $W_1 $ and $W_2$ of $T \times [0,1)$. In this case, as in \cite{Menasco} and \cite{ThurstonNotes}, we obtain a graph with ideal vertices on the two boundaries $T \times \{0\}$ such that when pairs of corresponding faces are appropriately identified, we obtain $M' \setminus L$. For each face $F_i$ on $\partial W_1$ and corresponding face $F_i'$  on $\partial W_2$, we take $F_i \times [0,1) \subset W_1$ and $F_i' \times [0,1) \subset W_2$ and glue them together appropriately along $F_i \times \{0\}$ and $F_i' \times \{0\}$, to obtain a bipyramid,  the ideal apexes of which correspond to $F_i \times \{1\}$ and  $F_i' \times \{1\}$. Then these ideal bipyramids glue together to yield the link complement.

In the case of $g \geq 2$, the argument is identical, only now we use two copies $W_1$ and $W_2$ of $S \times I$. Then for each pair of corresponding faces $F_i$ and $F_i'$, we obtain a truncated bipyramid, such that the truncation faces from the collection of bipyramids glue together to yield the two boundaries of $M'$. 
\end{proof}

Note that we could also have proven Lemma \ref{lem:bipdecomp} by arguing as in \cite{Adams2015} that each truncated octahedron from Lemma \ref{lem:octdecomp} can be decomposed into four pieces and the pieces can be reassembled to obtain the face-centered bipyramids.

\section{Hyperbolic polyhedra}\label{sec:polyhedra}

This section is devoted to the investigation of hyperbolic polyhedra and the relationship of their volumes with their dihedral angles. In Theorems \ref{thm:Ushijima} and \ref{thm:Schlafli}, we recall some well-known formulas for the volume of hyperbolic tetrahedra, which we then use in Proposition \ref{prop:maximalWedge} to find the maximal volume tetrahedron with certain constraints. From this, we deduce results about the maximal volume of certain generalized octahedra and bipyramids (Corollaries \ref{2v8ismaximaloct} and \ref{maximalBipyramid}, respectively), which in turn yield bounds on the volume of links in $S \times I$ in terms of crossing number (Corollaries \ref{cor:octvolbd} and \ref{cor:bipvolbd}).

\subsection{Generalized tetrahedra}

We begin by recalling the definition of hyperbolic polyhedra with ultra--ideal vertices. For a more thorough discussion, the reader is advised to consult \cite{Ushijima2006}.

A \emph{generalized hyperbolic tetrahedron} is the convex hull of four points which may be \emph{finite} (within $\mathbb{H}^3$), \emph{ideal} (on $\partial\mathbb{H}^3$), or \emph{ultra-ideal} (outside $\mathbb{H}^3 \cup \partial \mathbb{H}^3$).

Ultra-ideal points fit most naturally into the Klein ball model of $\mathbb{H}^3$: given an ultra-ideal point $x$ outside the unit sphere in $\mathbb{R}^3$, consider the cone of lines through $x$ tangent to the sphere. The {\it canonical truncation plane} associated to $x$ is the plane containing the circle where that cone intersects the sphere. Geodesic lines and planes through the ultra-ideal point are computed as Euclidean lines and planes through the point as usual, cut off at the truncation plane (see Figure \ref{ultraidealtruncation}). Note that all edges of the tetrahedron that passed through $x$ before truncation are perpendicular to the corresponding truncation plane.

We will often prefer to work in the Poincar\'e ball model for computing lengths and angles. Since the two models agree on the sphere at infinity, we can do this by using the Klein model to locate the ideal boundaries of geodesic lines and planes, and then construct the geodesics corresponding to those boundaries in the Poincar\'e model.

\begin{figure}[htbp]
\centering
\includegraphics[scale = .5]{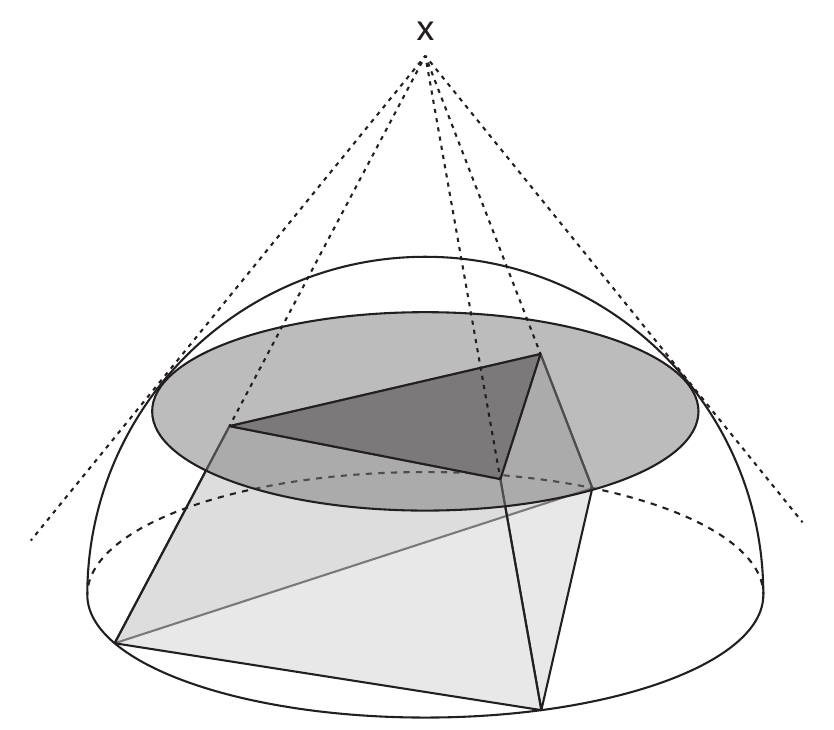}
\caption{The canonical truncation of an ultra-ideal point in the Klein ball model.}
\label{ultraidealtruncation}
\end{figure}

A generalized hyperbolic tetrahedron is fully determined by its six dihedral angles (\cite{Ushijima2006}). We can thus specify a tetrahedron with a vector $\Delta = (A, B, C, D, E, F) \in [0,\pi]^6$, with the dihedral angles labelled as in Figure \ref{fig:LabelledTetrahedron}. We restrict our attention to \emph{mildly truncated} tetrahedra, those in which truncation planes for distinct ultra-ideal vertices do not intersect within $\mathbb{H}^3$. It should be understood that all truncated tetrahedra in this paper are mildly truncated.

\begin{figure}[htbp]
\centering
\includegraphics[width = .4\textwidth]{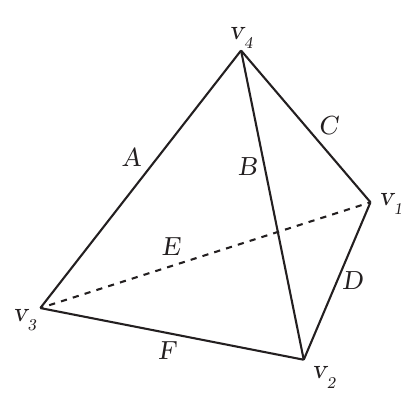}
\caption{A tetrahedron with vertices and dihedral angles labeled.}
\label{fig:LabelledTetrahedron}
\end{figure}

When the generalized tetrahedron is mildly truncated, there is a formula for its volume in terms of its dihedral angles.

\begin{thm}[\cite{Ushijima2006}]\label{thm:Ushijima}
Given a generalized hyperbolic tetrahedron $\Delta$ with dihedral angles as in Figure \ref{fig:LabelledTetrahedron}, let
\[a = e^{iA},\; b = e^{iB},\; ...,\; f = e^{iF},\]
\[G = 
\begin{pmatrix}
1 & -\cos A & -\cos B & -\cos F \\
-\cos A & 1 & -\cos C & -\cos E\\
-\cos B & -\cos C & 1 & -\cos D \\
-\cos F & -\cos E & -\cos D & 1
\end{pmatrix},\]
\[z_1 = -2 \frac{\sin A \sin D + \sin B \sin E + \sin C \sin F - \sqrt{\det G}}{ad + be + cf + abf + ace + bcd + def + abcdef},\]
\[z_2 = -2 \frac{\sin A \sin D + \sin B \sin E + \sin C \sin F + \sqrt{\det G}}{ad + be + cf + abf + ace + bcd + def + abcdef},\]
\[U(z, \Delta) = \frac{1}{2}(\text{Li}_2(z) + \text{Li}_2(abdez) + \text{Li}_2(acdfz) + \text{Li}_2(bcefz) - \text{Li}_2(-abcz) - \text{Li}_2(-aefz) - \text{Li}_2(-bdfz) - \text{Li}_2(-cdez)),\]
where $\text{Li}_2(z)$ is the dilogarithm function, defined by the analytic continuation of the integral
\[ \text{Li}_2(x) = - \int_0^x \frac{\log(1-t)}{t}\,dt\]
for $x \in \mathbb{R}_{>0}$. Then the volume of $\Delta$ is given by
\[\emph{vol}(\Delta) = \frac{1}{2} \text{Im}(U(z_1, \Delta) - U(z_2, \Delta)).\]
\end{thm}

The generalized volume formula in terms of the angles is unwieldy, but the total differential takes an elegant form in terms of the edge lengths. This is Schl\"afli's differential formula, in the three-dimensional case (see, \emph{e.g.} \cite{Schlafli}, \cite{Kellerhals}).
\begin{thm}[Schl\"afli's differential formula]\label{thm:Schlafli}
Given a generalized hyperbolic tetrahedron $\Delta$ with dihedral angles $\alpha_i$ and corresponding edge lengths $\ell_i$, the differential of the volume function is given by
\[d\emph{vol}(\Delta) = -\frac{1}{2} \sum_{i=0}^6 \ell_i \,d\alpha_i.\]
\end{thm}

\subsection{Maximal volume tetrahedra}

On a given link complement in any of the cases we have discussed, both Thurston's octahedral construction at the crossings  and the bipyramid construction in the faces yield octahedra and bipyramids that decompose into generalized tetrahedra with two ideal vertices.

\begin{prop}\label{prop:maximalWedge}
Let $\Delta$ be a tetrahedron labelled as in Figure \ref{fig:LabelledTetrahedron} with $v_1$ and $v_2$ ideal vertices. For fixed dihedral angle $A \in [0, \pi]$ on the edge between $v_3$ and $v_4$, the maximal volume for such a tetrahedron is achieved when the other angles are
\[D = \arccos\left( \frac{1}{2}(\cos A - 1)\right)\]
\[B = C = E = F = \frac{\pi - D}{2}.\]
In this case $v_3$ and $v_4$ are ultra-ideal.
\end{prop}

\begin{proof}
The condition that $v_1$ and $v_2$ are ideal vertices imposes the constraints
\[B + F + D = C + E + D = \pi.\] 
We use Lagrange multipliers to maximize the volume subject to those constraints.
 Let 
\[\begin{array}{ll}
g_1(B, C, D, E, F) &= B+F+D-\pi \\
g_2(B, C, D, E, F) &= C+E+D-\pi
\end{array}\]
Then the method of Lagrange multipliers with multiple constraints makes us consider solutions of
\[\nabla_{B, \ldots, F, \lambda, \mu} ( \text{vol}(A, \ldots, F) - \lambda g_1(B, \ldots, F) - \mu g_2(B, \ldots, F) ) = 0\]
which yields

\[\p{\text{vol}}{B} = \p{\text{vol}}{F} = \lambda\]
\[\p{\text{vol}}{C} = \p{\text{vol}}{E} = \mu\]
\[\p{\text{vol}}{D} = \lambda + \mu\]
By Schl\"afli's differential formula, those equations become
\begin{align*}
\ell_B &= \ell_F \\
\ell_C &= \ell_E \\
\ell_D &= \ell_B + \ell_C.
\end{align*}
However, all these edges have at least one ideal endpoint, so the lengths are infinite. To recover useful information from Schl\"afli's formula, we replace the ideal vertices with finite vertices $\bar{v}_1$ and $\bar{v}_2$ and consider the limiting behavior as they approach the ideal points $v_1$ and $v_2$. That is, we require
\[\lim_{\bar{v}_2 \to v_2} (\ell_B - \ell_F) = \lim_{\bar{v}_1 \to v_1} (\ell_C - \ell_E) = \lim_{\substack{\bar{v}_1 \to v_1 \\ \bar{v}_2 \to v_2}} (\ell_D - \ell_B - \ell_C) = 0.\]

Up to isometries of $\mathbb{H}^3$ in the Poincar\'e ball model, we may let $v_1 = (0,0,-1)$, $v_2 = (0,0,1)$, and $v_4 = (r_4,0,0)$, as in Figure \ref{Fig:maximalwedge}. Then $v_3$ lies at some point
\[v_3 = (r_3 \cos D \sin\phi, r_3 \sin D \sin\phi, r_3 \cos\phi).\]

\begin{figure}[htbp]
\centering
\includegraphics[scale=.6]{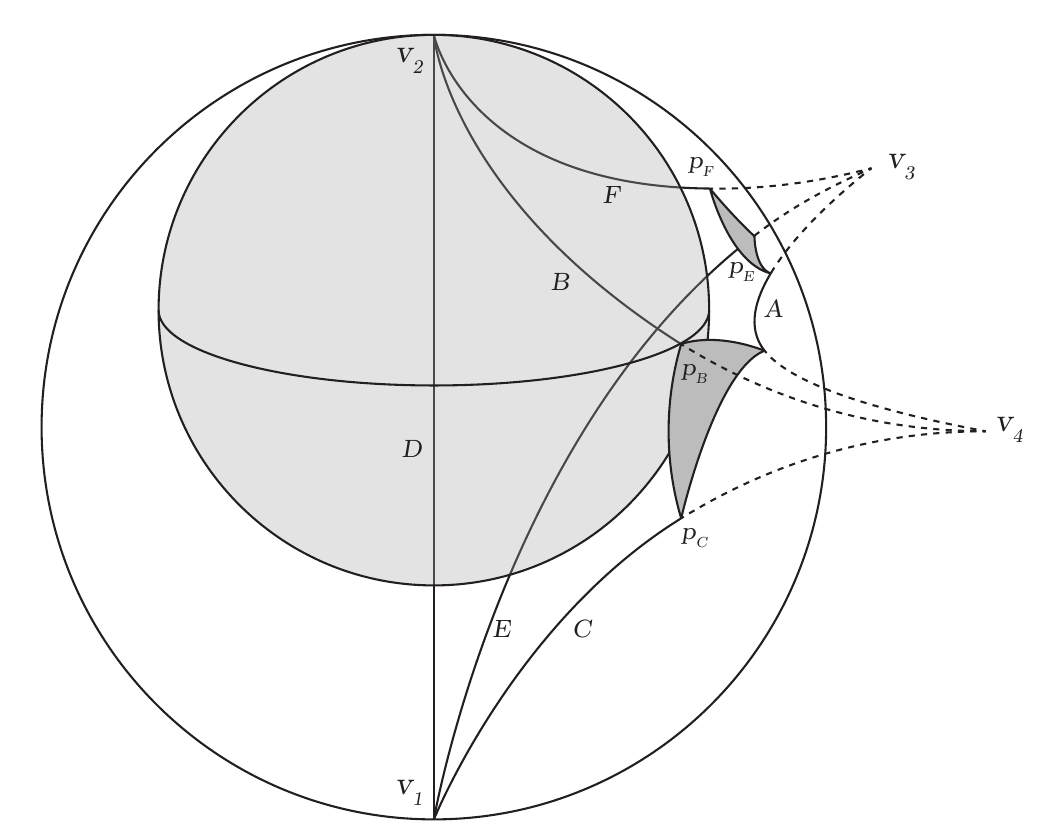}
\caption{A truncated tetrahedron as in Proposition \ref{prop:maximalWedge}.}
\label{Fig:maximalwedge}
\end{figure}

Let $p_B, p_C, p_E, p_F$ be the endpoints of edges B, C, E, and F respectively, opposite from $v_1$ and $v_2$. That is, if $v_4$ is finite or ideal then $p_B = p_C = v_4$, but if $v_4$ is ultra-ideal then $p_B$ and $p_C$ are the intersections of their respective edges with the truncation plane for $v_4$, and similarly for $p_E$ and $p_F$. The point $p_B$ lies on a unique horosphere centered at $v_2$, given by
\[x^2 + y^2 + (z - (1 - R_B))^2 = R_B^2\]
for some $R_B \in [0,1]$. By symmetry, $p_C$ lies on an opposite horosphere of the same radius centered at $v_1$. Similarly $p_E$ and $p_F$ lie on horospheres of radii $R_E$ and $R_F$ around $v_1$ and $v_2$ respectively. As $\bar{v}_1$ and $\bar{v}_2$ approach $v_1$ and $v_2$, $\ell_B - \ell_F$ limits to the finite distance between the concentric horospheres of radii $R_B$ and $R_F$, and similarly $\ell_C - \ell_E$ limits to the distance between concentric horospheres of radii $R_B$ and $R_E$. Furthermore, since we can split edge D at the origin, $\ell_D - \ell_B - \ell_C$ tends to twice the distance between concentric horospheres of radii $R_B$ and $\frac{1}{2}$. Thus the optimization conditions become
\[R_E = R_F = R_B = \frac{1}{2}.\]

This is satisfied if $v_3$ and $v_4$ lie at the origin, but then $\Delta$ is degenerate with volume 0.  Alternatively the truncation planes for $v_3$ and $v_4$ must both be mutually tangent to the two horospheres of radius $\frac{1}{2}$ centered at $v_1$ and $v_2$. Therefore $v_3$ must be equidistant from the Euclidean centers of the horospheres, which lie at  $(0, 0, \pm \frac{1}{2})$, so it must lie on the $xy$-plane. Thus $\phi = \pi/2$. At this point we can see by symmetry of the tetrahedron and the ideal vertex constraints that 
\[B = C = E = F = \frac{\pi - D}{2}.\]
The Euclidean radius of the truncation plane for $v_3$ is $\sqrt{r_3^2 - 1}$, so  by the Pythagorean theorem we have
\[\left(\sqrt{r_3^2 - 1} + \frac{1}{2}\right)^2 = r_3^2 + \left(\frac{1}{2}\right)^2\]
which implies $r_3 = \sqrt{2}$ and by the same reasoning, $r_4 = \sqrt{2}$. Thus, we see that $v_4 = (\sqrt{2},0,0)$ and $v_3 = (\sqrt{2}\cos D, \sqrt{2}\sin D, 0)$. 

Finally we compute an explicit relationship between $A$ and $D$. The top and bottom faces of this tetrahedron are given by
\[\left(x - \frac{1}{\sqrt{2}}\right)^2 + \left(y - \frac{1 - \cos D}{\sqrt{2}\sin D}\right)^2 + (z \pm 1)^2 = \frac{1 - \cos D}{\sin^2 D},\]
and the angle between them is
\[A = \arccos(1 + 2\cos D).\]
Inverting the function, we arrive at
\[D = \arccos\left(\frac{1}{2}(\cos A - 1)\right).\]

\end{proof}

\begin{coro}\label{v8/2ismaxwedge}
The maximal volume generalized tetrahedron with two ideal vertices has volume $v_\text{oct}/2$, with angles
\[A = 0, \quad D = \frac{\pi}{2}, \quad B = C = E = F = \frac{\pi}{4}.\]
\end{coro}

\begin{proof}
We see immediately from the Schl\"afli formula that decreasing any one angle increases the volume. Thus among tetrahedra with two ideal vertices as described in Proposition \ref{prop:maximalWedge}, the maximal volume occurs when $A = 0$. The other angles follow from the formulas in the lemma. The volume of this tetrahedron is $v_\text{oct}/2$, which can be seen by gluing four copies around the central edge with angle $D = \pi/2$ to obtain the polyhedron in Figure \ref{maximaloct}. By chopping off eight tetrahedra from this polyhedron, each with one finite vertex and three ideal vertices, we are left with a single ideal regular octahedron. The eight tetrahedra we chopped off can be reassembled around the finite vertex to create a second ideal regular octahedron, hence the polyhedron has volume $2 v_\text{oct}$ and the original tetrahedron has volume $v_\text{oct}/2$.
\end{proof}


\begin{coro}\label{2v8ismaximaloct}
The maximal volume of a generalized hyperbolic octahedron with two opposite ideal vertices is $2v_\text{oct}$.
\end{coro}

\begin{proof}
As mentioned, a generalized hyperbolic octahedron with two opposite ideal vertices can be cut into four tetrahedral wedges around the core line connecting the ideal vertices. All four of these wedges may be the maximal tetrahedron described in Corollary \ref{v8/2ismaxwedge}, yielding the generalized octahedron as shown in Figure \ref{maximaloct}. By decomposing  as in Figure \ref{maxoctexploded} and then recomposing, we can turn this into two ideal regular octahedra with volume $2v_{oct}$.
\end{proof}

\begin{figure}[htbp]
\centering
\includegraphics[scale=0.3]{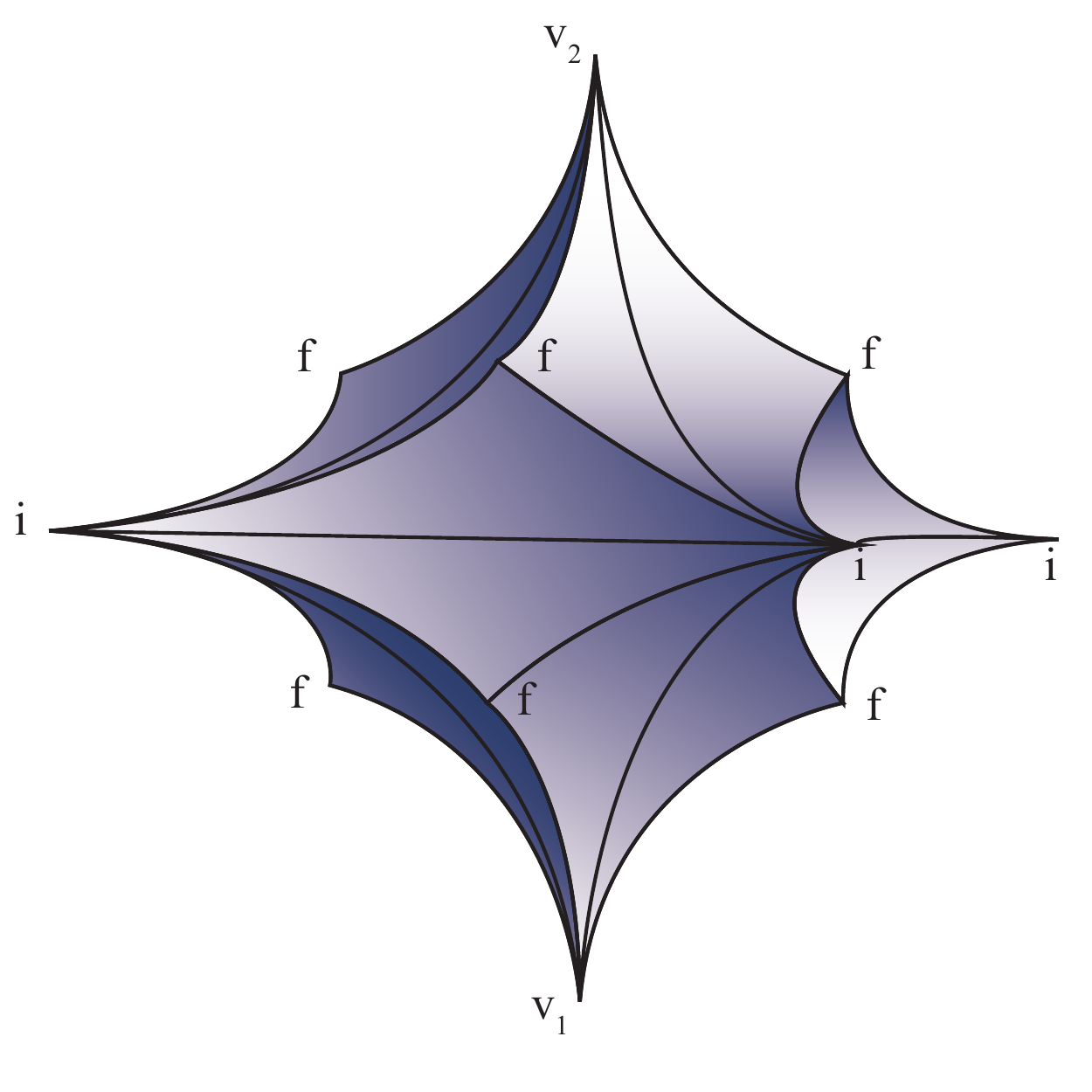}
\caption{The maximal volume generalized hyperbolic octahedron with two opposite ideal vertices. The other four vertices of the octahedron are ultra-ideal, with length 0 edges between the adjacent truncated faces, yielding additional ideal vertices, each labelled with an $i$. Vertices labelled with an $f$ are finite vertices on the truncated faces.}
\label{maximaloct}
\end{figure}

\begin{figure}[htbp]
\centering
\includegraphics[scale=0.3]{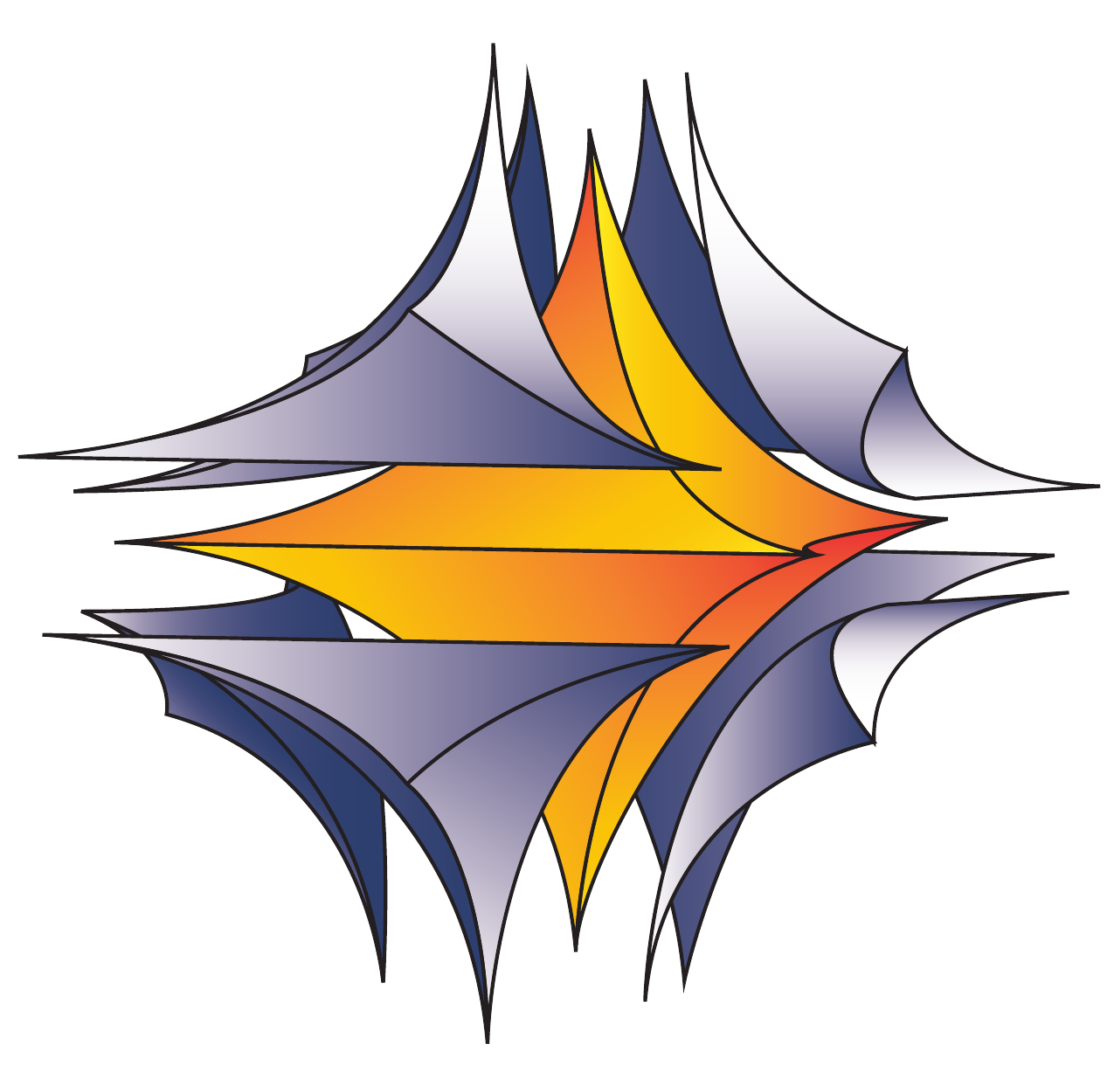}
\caption{The maximal volume generalized octahedron decomposes into one ideal regular octahedron and eight tetrahedra that recompose around their finite vertices into a single additional ideal regular octahedron.}
\label{maxoctexploded}
\end{figure}

From this bound together with the generalized octahedral decomposition described in Lemma \ref{lem:octdecomp}, we can immediately deduce

\begin{coro}\label{cor:octvolbd}
If $L$ is a hyperbolic link in $S_g \times I$,  where $g \geq 2$, then $vol(S_g \times I \setminus L ) \leq c(L) 2v_\text{oct}$.
\end{coro}
\begin{proof}
By Lemma \ref{lem:octdecomp}, $S_g \times I \setminus L$ has a decomposition into $c(L)$ generalized octahedra, which by Lemma \ref{cor:octvolbd} each have volume at most $2v_{\text{oct}}$.
\end{proof}

Proposition \ref{prop:maximalWedge} also allows us to determine the maximal volume bipyramids with ideal equatorial vertices and ultra-ideal apexes. 

\begin{coro}
\label{maximalBipyramid}
The maximal volume generalized $n$-bipyramid with ideal vertices around the central polygon is made of $n$ identical maximal volume tetrahedra of the type described in Proposition \ref{prop:maximalWedge}, with angles $A = 2\pi / n$.
\end{coro}

\begin{proof}
A generalized $n$-bipyramid with ideal vertices around the central polygon can be cut along its core line into tetrahedra $\Delta_1, ..., \Delta_n$, each with two ideal vertices. Label the edges of each as in Figure \ref{Fig:maximalwedge}, with subscripts to distinguish the tetrahedron to which they belong. The ideal vertices on each tetrahedron dictate the constraints
\[B_i + F_i + D_i = C_i + E_i + D_i = \pi,\]
and the fact that the tetrahedra are all glued together along the core line additionally requires
\[ \sum_{i = 1}^n A_i = 2\pi. \]
We again maximize using Lagrange multipliers and substitute from Schl\"afli's formula, obtaining
\begin{align*}
\ell_{A_i} &= \ell_{A_j} \\
\ell_{B_i} &= \ell_{F_i} \\
\ell_{C_i} &= \ell_{E_i} \\
\ell_{D_i} &= \ell_{B_i} + \ell_{C_i}
\end{align*}
for all $i, j$. By the same logic used in the proof of Proposition \ref{prop:maximalWedge}, the last three equations imply that $\Delta_i$ is isometric to a tetrahedron in the Poincar\'e ball model with vertices
\[v_1^i = (0,0,-1),\quad v_2^i = (0,0,1),\quad v_4^i = (\sqrt{2},0,0),\quad v_3^i = (\sqrt{2}\cos D_i, \sqrt{2} \sin D_i, 0). \]
In this model it is apparent that $\ell_{A_i}$ increases monotonically with $D_i$. Thus $\ell_{A_i} = \ell_{A_j}$ implies $D_i = D_j$, and consequently $A_i = A_j = 2\pi / n$. The other angles depend on $A_i$ as in the lemma because the tetrahedra are maximal.
\end{proof}

We refer to the bipyramids described in Corollary \ref{maximalBipyramid} as \emph{maximal doubly truncated $n$-bipyramids}, denoted $B_n^\text{trunc}$. Some volumes of these bipyramids are shown in Figure \ref{bipyramidvolumetable}. We note that the ratio
\[\frac{\text{vol}(B_n^\text{trunc})}{n}\]
 is strictly increasing with $n$, asymptotically approaching $v_\text{oct} / 2$ as the tetrahedral wedges making up the bipyramids approach the maximal wedge (with angle $A = 0$) discussed in Corollary \ref{v8/2ismaxwedge}. This contrasts with the case of ideal bipyramids, where 
\[\lim_{n \to \infty} \frac{\text{vol}(B_n^\text{ideal})} {n} = 0,\] 
peaking when $n = 6$ at $v_\text{tet}$, the volume of a regular ideal tetrahedron \cite{Adams2015}.

As in Corollary \ref{cor:octvolbd}, we can also obtain an upper bound on volume in terms of the bipyramidal decomposition of Lemma \ref{lem:bipdecomp}.

\begin{coro}\label{cor:bipvolbd}
If $L$ is a hyperbolic link in $S_g \times I$,  where $g \geq 2$, then 
\[vol(S_g, \times I \setminus L ) \leq \sum_{\{F_i\}} B_{n_i}^\text{trunc}\] where each maximal doubly truncated bipyramid corresponds to a unique non-bigon face $F_i$ with $n_i$ edges in the projection of $L$ to $S_g$.
\end{coro}
\begin{proof}
By Lemma \ref{lem:bipdecomp}, the link exterior has a decomposition into truncated bipyramids $B_{n_i}$, each of which corresponds to a non-bigon face. By Corollary \ref{maximalBipyramid}, each has volume at most $B_{n_i}^\text{trunc}$.
\end{proof}

\begin{figure}[ht]
\centering
\begin{subfigure}[b]{.48 \textwidth}
\centering
\includegraphics[scale=.5]{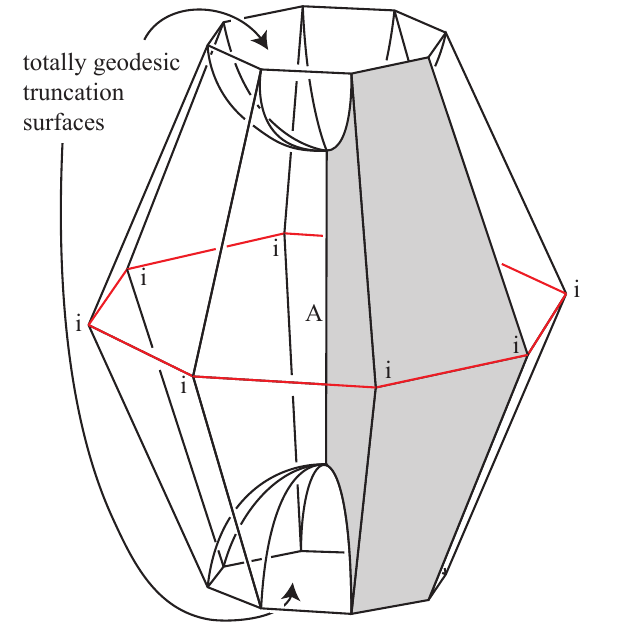}
\caption{A doubly truncated 8-bipyramid with angle $A=\frac{\pi}{4}$.}
\label{truncatedbipyramid}
\end{subfigure}
\begin{subfigure}[b]{.5 \textwidth}
\centering
\begin{tabular}{|| c | c ||}
\hline
$n$ & vol$(B_n^\text{trunc})$ \\
\hline \hline
2 & 0 \\
\hline
3 & 2.6667 \\
\hline
4 & 5.0747 \\
\hline
5 & 7.3015 \\
\hline
6 & 9.4158 \\
\hline
7 & 11.4580 \\
\hline
8 & 13.4520 \\
\hline
9 & 15.4122 \\
\hline
10 & 17.3481 \\
\hline
100 & 183.0944 \\
\hline
1000 & 1831.9213 \\
\hline
\end{tabular}
\caption{Volumes of maximal doubly truncated $n$-bipyramids.}
\label{bipyramidvolumetable}
\end{subfigure}
\caption{Maximal doubly truncated bipyramids.}
\end{figure}

\section{Computing volumes of tiling links}
\label{sec:volumes}

In this section we provide a method to compute the volumes for the exteriors of all alternating $k$--uniform tiling links (Theorem \ref{computingvolumes}). In order to describe the hyperbolic structure on these link complements, we first describe some basic building blocks of the appropriate geometric $n$-bipyramids, and explain how the symmetry of the tiling descends to symmetries of the bipyramids (Lemmas \ref{lem:eq_hyp}--\ref{lem:link_poly}).

To that end, define the wedge $\Delta_n(\theta)$ to be the hyperbolic tetrahedron with dihedral angles 
\begin{eqnarray}\label{eqn:bi_angles}
D=\pi - \theta \nonumber \\
B=C=E=F= \theta/2 \\
A = 2 \pi/n \nonumber
\end{eqnarray}
where the tetrahedron is labeled as in Figure \ref{fig:LabelledTetrahedron}.

We can determine the type (finite, ideal, ultra--ideal) of each of the vertices of $\Delta_n(\theta)$ from the sums of the dihedral angles of the edges meeting at the vertex. For example, since $B+C+D = E+F+D = \pi$, we know that the link of each of $v_1$ and $v_2$ is Euclidean, and therefore the vertices must both be ideal.

Similarly, we can determine the types of $v_3$ and $v_4$ from the value $ A + B+C$. If $A+B+C>\pi$, then the link around $v_4$ has positive curvature and hence $v_4$ must be finite (and since $A+B+C = A+E+F$, so must $v_3$). If $A+B+C=\pi$ (respectively $< \pi$) then the link of $v_4$ has zero (respectively negative) curvature, and so the vertex must be ideal (respectively ultra-ideal).

We can simplify this criterion using \eqref{eqn:bi_angles}: since $B+C = \theta$, we get that $v_3$ and $v_4$ are
\begin{equation}
\left. \begin{array}{r}
\text{finite} \\
\text{ideal} \\
\text{ultra-ideal}
\end{array} \right\}
\text{ if }
\left\{ \begin{array}{l}
\theta > A = (n-2)\pi/n \\
\theta = (n-2)\pi/n \\
\theta <  (n-2)\pi/n.
\end{array}\right. 
\end{equation}

Observe that $\Delta_n(\theta)$ admits a reflective symmetry that interchanges the $ACE$ and $ABF$ faces while fixing $A$ and reversing the orientation of $D$, so the $ACE$ and $ABF$ faces are isometric. Likewise, it admits a reflective symmetry that interchanges the $DBC$ and $DEF$ faces while reversing the orientation of $A$.

Gluing $n$ copies of $\Delta_n(\theta)$ together, $ACE$ face to $ABF$ face so the $A$ edges all are identified,  produces an $n$--bipyramid which we call the {\em symmetric $n$--bipyramid $B_n(\theta)$ of vertical angle $\theta$}. The {\it vertical edges} are defined to be  the edges that have an endpoint at either the  top or bottom apex, which can be finite, ideal or ultra-ideal. The remaining edges, always with both endpoints ideal, are called {\it equatorial edges} and their vertices are called {\it equatorial vertices}.

Since the fixed points of the $A$--reversing involutions of each wedge all coincide after the gluing, it is easy to observe the following.

\begin{lem}\label{lem:eq_hyp}
The equatorial vertices and edges of $B_n(\theta)$ all lie on the closure of a hyperplane.
\end{lem}
\begin{proof}
This follows from the symmetry of each of the wedges $\Delta_n(\theta)$, which fit together to exhibit a symmetry of the bipyramid itself. The fixed locus of a reflection in hyperbolic 3--space is a hyperbolic plane.
\end{proof}

We may also use the symmetries of $B_n(\theta)$ to describe the links of its vertices.

\begin{lem}\label{lem:link_rhomb}
Let $v$ be an equatorial vertex of $B_n(\theta)$ and $H$ be a small enough horosphere centered at $v$. Then $H \cap B_n(\theta)$ is a (Euclidean) rhombus with angles $\theta$ and $\pi - \theta$.
\end{lem}
\begin{proof}
Since four edges of $B_n(\theta)$ meet at each equatorial (ideal) vertex $v$, it is immediate that the link of the vertex is a (Euclidean) quadrilateral $Q$. Since the dihedral angles of each edge are specified in our construction of $B_n(\theta)$, we see that the angles of $Q$ are $\theta$, $\pi - \theta$, $\theta$, and $\pi - \theta$, as desired. Therefore $Q$ is a parallelogram.

Now observe that $B_n(\theta)$ has a symmetry which interchanges the two vertical edges incident to $v$ while fixing the equatorial, as well as a symmetry which fixes the vertical edges while interchanging the equatorial. These symmetries demonstrate that the symmetry group of $B_n(\theta)$ acts transitively on the edges of $Q$, so its edges all have the same length. Therefore $Q$ is a rhombus.
\end{proof}

\begin{lem}\label{lem:link_poly}
The link of either of the apex vertices of $B_n(\theta)$ is a regular $n$--gon with interior angles $\theta$.
\end{lem}
\begin{proof}
Since there are $n$ edges which meet at an apex, each with dihedral angle $\theta$, the link of an apex must be an $n$--gon $P$ whose interior angles are all $\theta$. As $B_n(\theta)$ is constructed from $n$ isometric wedges $\Delta_n(\theta)$, it has an order $n$ rotational symmetry and in particular its symmetry group acts transitively on the edges of $P$, thus $P$ is regular.
\end{proof}

Notice that when the apex of $B_n(\theta)$ is ultra-ideal (that is, when $\theta <  (n-2)\pi/n$) then there is a canonical geometric realization of its link,
namely, the intersection of the bipyramid with the unique truncation plane $\mathcal{P} \cong \mathbb{H}^2$ perpendicular to the $n$ edges ending at the apex.
Since there is at most one regular hyperbolic $n$--gon of given interior angle (up to isometry),
this implies that $n$ and $\theta$ completely determine the polygon $P = B_n(\theta) \cap \mathcal{P}$
(which forms the upper or lower face of the truncated bipyramid) up to isometry.

Similarly, suppose that the apex of $B_n(\theta)$ is finite (that is, $\theta > (n-2)\pi/n.$); then likewise there is a unique regular spherical $n$--gon of given interior angle (up to dilation and rotation of the sphere). Therefore for any $\epsilon>0$ such that the ball $B_\epsilon$ about the apex does not meet the equatorial hyperplane (Lemma \ref{lem:eq_hyp}), we see that the polygon $P_\varepsilon = B_n(\theta) \cap \partial B_\epsilon$ is uniquely determined by $n$, $\theta$, and $\varepsilon$.

We will now prove that the topological decomposition of $\Ext(L)$ into generalized bipyramids can be realized geometrically to obtain a hyperbolic structure on the complement.

\begin{thm}\label{computingvolumes}
Let $\TT$ be a $k$-uniform tiling of $\mathbb{X} = \mathbb{S}^2, \mathbb{E}^2$ or $\mathbb{H}^2$ by equilateral polygons that is either 3--regular,  or 4--regular, or is made up of 3-valent and 4-valent vertices such that there is a subcollection of edges that pair up the 3-valent vertices.  Let $\Gamma$ be a torsion--free subgroup of the symmetry group of $\TT$ with compact fundamental domain. Set $S = \mathbb{X} / \Gamma$ so that $S$ is tiled by tiles $\{F_i\}$, each a regular $n_i$--gon of angle $\alpha_i$. Let $L \subset S \times I$ be an alternating tiling link associated to $\TT$.

Then the complete hyperbolic structure on $\Ext(L)$ (with totally geodesic boundary when $\mathbb{X} = \mathbb{H}^2$)
may be obtained by gluing together symmetric hyperbolic bipyramids, each corresponding to a face of the tiling of $S$. That is,
\[\Ext(L) = \bigsqcup_{\{F_i\}} B_{n_i}(\alpha_i) / \sim\]
yields the complete hyperbolic structure on $\Ext(L)$.

\end{thm}

\begin{proof}
Note that the uniqueness statement implicit in the Theorem is a consequence of Theorem \ref{uniquetotgeostructure}.

By Lemma \ref{lem:bipdecomp},
we already know that $\Ext(L)$ can be topologically realized by gluing together bipyramids of the given combinatorial types (number of sides and finite/ideal/ultra-ideal) and that the links of the apexes fit together to give a (topological) tiling of the appropriate surface. To see that the hyperbolic structures given by $B_{n_i}(\alpha_i)$ glue together to give the complete hyperbolic structure on $\Ext(L)$ (with totally geodesic boundary when $\mathbb{X} = \mathbb{H}^2$), we will explicitly show that the bipyramids satisfy certain gluing conditions, which by the Poincar\'{e} Polyhedron Theorem will imply the result. (See \cite{Epstein1994} for instance).

In particular, we first demonstrate that the bipyramids $B_{n_i}(\alpha_i)$ glue together along isometric faces and that the dihedral angles around each edge fit together to yield a total angle of $2\pi$. Moreover, the links of the equatorial vertices fit together to induce a Euclidean structure on each of the cusps of $\Ext(L)$. Ultimately, this implies that the hyperbolic structure on $\Ext(L)$ obtained from gluing together the hyperbolic structures on the bipyramids $B_{n_i}(\alpha_i)$ is actually a complete hyperbolic structure (with totally geodesic boundary if appropriate), as in \cite{ThurstonNotes} \S 3.10.

\para{Isometric gluings} The fact that the bipyramids $B_{n_i}(\alpha_i)$ are always glued along isometric faces follows {\em a fortiori} from the statement below:

\begin{claim}
Any two faces of any two of the $B_{n_i}(\alpha_i)$ are isometric, so long as neither of the faces is the truncation face obtained by the truncation of an ultra-ideal vertex.
\end{claim}
\begin{proof}[Proof of Claim]
Suppose first that $\mathbb{X} = \mathbb{S}^2$. Choose some $\varepsilon$ small enough so that the $\varepsilon$ neighborhood $B_\varepsilon$ about any apex $v_i$ of any of the $B_{n_i}(\alpha_i)$ does not meet the equatorial plane (see Lemma \ref{lem:eq_hyp}). Then by the discussion following Lemma \ref{lem:link_poly}, we know that $\varepsilon$, $n_i$ and $\alpha_i$ together completely determine the polygon of intersection $P_{i, \varepsilon} := B_{n_i}(\alpha_i) \cap \partial B_\varepsilon$.

Using standard spherical trigonometry, the length of one edge of $P_{i, \varepsilon}$ is enough to determine the angle subtended by the sides of a face of $B_{n_i}(\alpha_i)$ which meet at $v$. In turn, this angle completely determines the isometry class of the face, which is a hyperbolic triangle with two ideal vertices and one finite vertex.

Now the angles $\alpha_i$ are chosen so that the spherical tiling $\TT$ is equilateral, and so we see that for any two given $B_{n_i}(\alpha_i)$ and $B_{n_j}(\alpha_j)$, the corresponding polygons $P_{i, \varepsilon}$ and $P_{j, \varepsilon}$ have sides of equal length. Therefore the angles subtended at the apexes $v_i$ and $v_j$ by the sides of a face are equal, hence the faces themselves are isometric.  See Figure \ref{finiteapexes}.

\begin{figure}[ht]
\centering
\includegraphics[scale=.6]{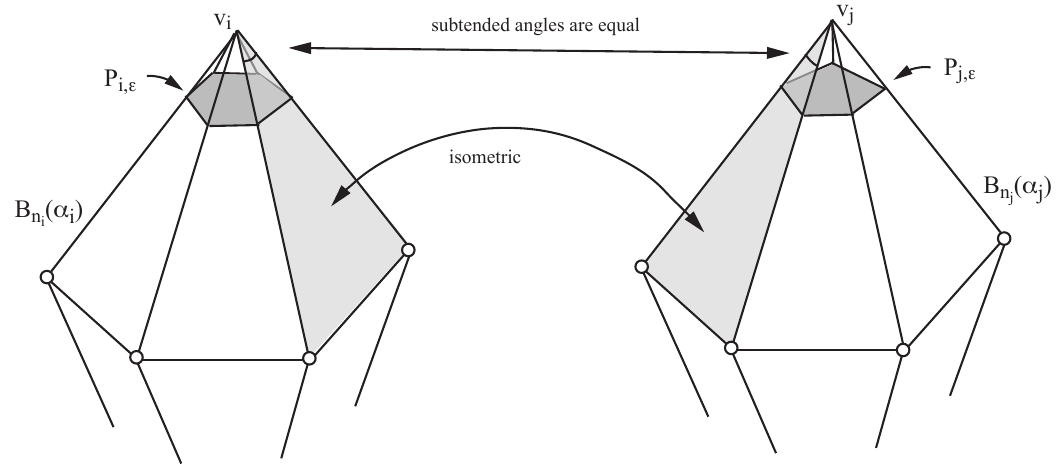}
\caption{Gluing isometric faces of bipyramids with finite apexes.}
\label{finiteapexes}
\end{figure}

If $\mathbb{X} = \mathbb{E}^2$, then the apexes of each of the bipyramids $B_{n_i}(\alpha_i)$ are ideal and hence the faces of the bipyramids are all ideal triangles. Every ideal triangle is isometric, and so the claim is proven.

Finally, when $\mathbb{X} = \mathbb{H}^2$ (and hence the apexes are all ultra-ideal), we recall from Lemma \ref{lem:link_poly} and the discussion which follows it that $n_i$ and $\alpha_i$ completely determine the isometry type of the polygons $P_i$ coming from the intersection of the bipyramids with their truncation planes.

Now the non--$P_i$ faces of each $B_{n_i}(\alpha_i)$ are hyperbolic quadrilaterals with angles $0, 0, \pi/2, \pi/2$ and one finite side, which it shares with a $P_i$ face. It is an easy exercise to show that the length of the finite side completely determines such a quadrilateral up to isometry, and so the lengths of the sides of the truncation polygons $P_i$ determine the isometry type of the non--$P_i$ faces. But now since the angles $\alpha_i$ are such that the tiling $\TT$ is equilateral, the side lengths of the $P_i$ are all the same, and hence the non--$P_i$ faces of the bipyramids $B_{n_i}(\alpha_i)$ are all isometric.
\end{proof}

\para{Angle sum, non-bigon case}
In order to verify the rest of the necessary gluing conditions, we first consider tilings that are 4--regular and therefore require no insertions of bigons.
Around a crossing, the edge labels on the adjacent bipyramids appear as in Figure \ref{bipyramidgluing}.

\begin{figure}[htbp]
\centering
\includegraphics[scale= 0.7]{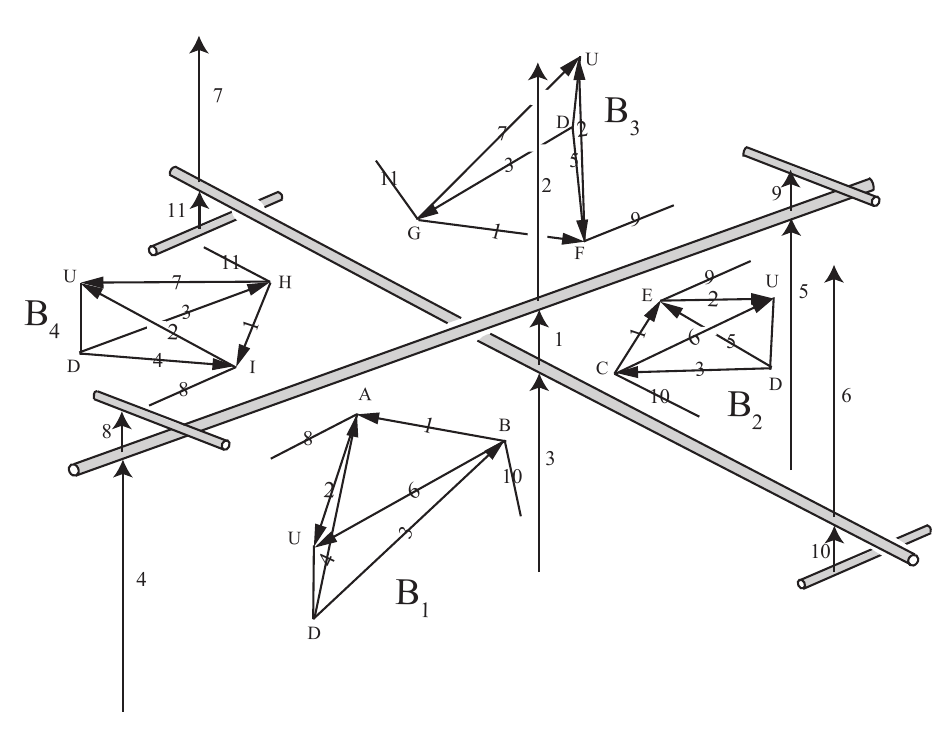}
\caption{Gluings of bipyramids around a crossing.}
\label{bipyramidgluing}
\end{figure} 

There are four bipyramids which contribute one edge each to each vertical edge class (labeled by $2$ and 3 in Figure \ref{bipyramidgluing}), and these four edges each have dihedral angles $\alpha_1, \alpha_2, \alpha_3, \alpha_4$, equal to the angle of the corresponding polygon's vertex in the (equilateral realization of the) tiling that generates $L$. In particular, since the links of the apexes fit together to form a tiling, we must have that $\alpha_1 + \alpha_2+ \alpha_3 + \alpha_4 = 2\pi$. Therefore the dihedral angles around any vertical edge class add up to $2\pi$ as desired and the links of the vertices at the top of edge 2 and bottom of edge 1 fit together correctly since they are part of the tiling from which the bipyramid dihedral angles came.

We now consider the equatorial edges.
Each edge class of equatorial edges corresponds to an edge in the link complement at one of the crossings, going from the undercrossing to the overcrossing (edge 1 in Figure \ref{bipyramidgluing}
). For each such edge class, there are again four bipyramids of dihedral angles $\beta_j$ for $j = 1, \ldots, 4$ which contribute one edge each. However, by construction of the $B_{n_i}(\alpha_i)$, we know that each $\beta_j = \pi - \alpha_j$ where $\alpha_j$ are the angles from the preceding paragraph. Therefore since $\alpha_1 + \alpha_2+ \alpha_3 + \alpha_4 = 2\pi$, we have that
\[\beta_1 + \beta_2 + \beta_3 + \beta_4 = (\pi - \alpha_1) + (\pi - \alpha_2)+(\pi - \alpha_3) + (\pi - \alpha_4) = 4 \pi - 2\pi = 2\pi\]
and so the sum of dihedral angles about any equatorial edge class is also $2\pi$.

Because we are using symmetric bipyramids, the link of each ideal equatorial vertex is a rhombus. Hence, when we glue the rhombi around a vertex corresponding to the crossection of an edge, the edge of the last rhombus matches in length with the length of the edge of the first rhombus, to which it is to be glued, as in Figure \ref{cuspnobigon}. In particular, this means there is no shearing along the edges as we generate isometires by gluing on subsequent bipyramids around the edge class. 

\para{Cusp structure, non-bigon case}
We now check that the gluings as described yield a complete structure on the cusps corresponding to the link complement. As mentioned, the link of each ideal equatorial vertex of a bipyramid is a rhombus, and those rhombi glue together around each vertex to yield exactly $2 \pi$ of angle. 

\begin{figure}[htbp]
\centering
\includegraphics[scale=0.5]{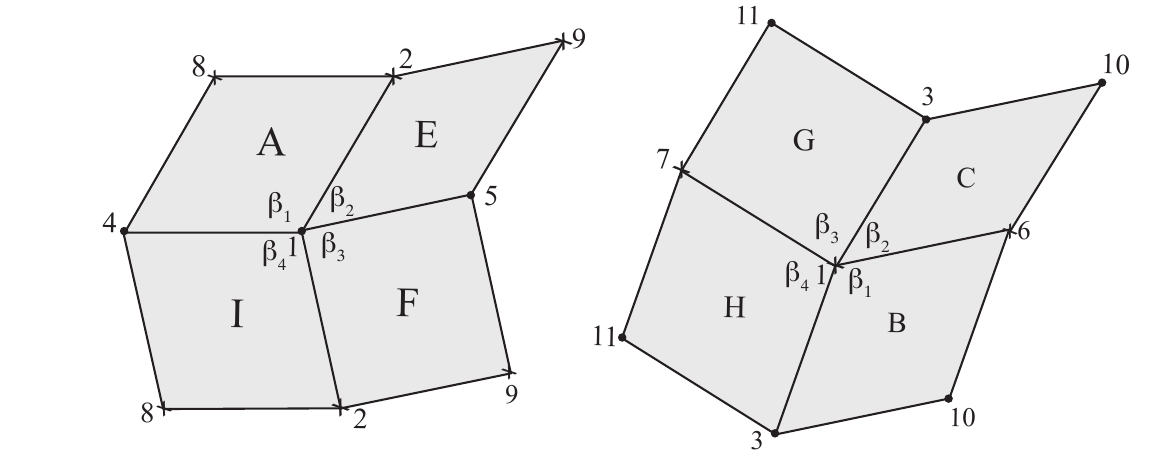}
\caption{Rhombi from each ideal vertex gluing together to create the cusp diagram at each crossing.}
\label{cuspnobigon}
\end{figure}

But then, as in Figure \ref{cuspnobigon}, the eight rhombi fit together in groups of four each around the top strand and around the bottom strand of the crossing. In each case, the top pair of edges of the resultant planar octagons glues to the bottom pair of edges by  Euclidean translation. This Euclidean transformation corresponds to the holonomy of the meridian of the link.  Since the rhombi fit together nicely around each vertex corresponding to the edges, we see that the holonomy of the longitude is also forced to correspond to a Euclidean transformation; this follows because we can take a quadrilateral that is tiled by copies of the rhombi such that it forms the basis for the developing map of the meridian and longitude actions. Once we know that the meridian corresponds to a translation, two opposite edges of the quadrilateral are parallel and of the same length. This forces the other two edges to have the same relationship and therefore the longitude is also a translation.  Hence, the cusp has a complete Euclidean structure. 
\bigskip

\para{Angle sum, bigon case}
We now consider the case when $\TT$ has some 3--valent vertices, with all such paired up by a subcollection of edges. In this situation we add bigon faces along each such edge. Call the two tiles meeting along the bigon edge $F_1$ and $F_3$, and let $F_2$ and $F_4$ denote the two tiles at the ends of the bigon edge. Since the sum of the angles around any vertex is $2\pi$, we know that the angles $\alpha_i$ of the $F_i$ satisfy the equations
\begin{align}\label{eqn:sum2pi_bigon}
\begin{split}
\alpha_1 + \alpha_2 + \alpha_3 & = 2\pi \\
\alpha_1 + \alpha_3 + \alpha_4 & = 2\pi
\end{split}
\end{align}
and hence $\alpha_2 = \alpha_4$.

For ease of notation, let $B_i$ denote the bipyramids $B_{n_i}(\alpha_i)$ corresponding to $F_1$ through $F_4$. By construction, the vertical edges of each $B_i$ have dihedral angles $\alpha_i$ and the equatorial edges have dihedral angles $\beta_i = \pi - \alpha_i$. Then it is immediate from \eqref{eqn:sum2pi_bigon} that the vertical edges glue up to provide an angle of $2\pi$ for the corresponding edge classes.

The equatorial edges which run from the undercrossings to the overcrossings at each end of the bigon are isotopic to one another, so we collapse them down to a single edge as in \cite{Menasco, AR}. The collapsed edge is labeled by 1 in Figure \ref{bigongluing}.

\begin{figure}[htbp]
\centering
\includegraphics[scale=0.5]{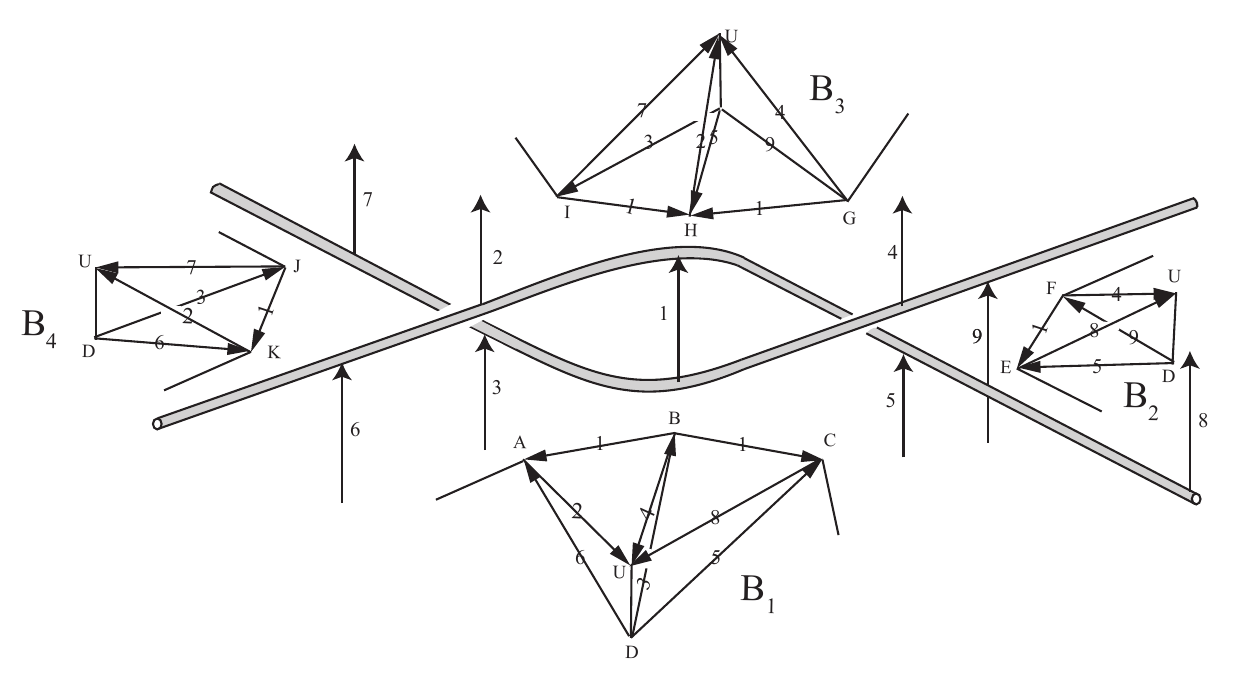}
\caption{The gluing of bipyramids around a bigon.}
\label{bigongluing}
\end{figure}

This means that $B_1$ and $B_3$ each contribute two equatorial edges to this edge class, and $B_2$ and $B_4$ each contribute one equatorial edge. Therefore the total dihedral angle around this edge class is
\[2 \beta_1 + 2 \beta_3 + \beta_2 + \beta_4 = 6 \pi - 2 \alpha_1 - 2 \alpha_3 - \alpha_2 - \alpha_4 = 6 \pi - 4 \pi = 2\pi,\]
where the first equation follows from the definition of the $\beta_i$ and the second follows from \eqref{eqn:sum2pi_bigon}. Therefore the sum of the dihedral angles about every edge is $2\pi$, as desired. Note in addition that since $\alpha_2 = \alpha_4$, we also have $\beta_1 + \beta_2 + \beta_3 = \pi$.

\para{Cusp structure, bigon case}
In order to show that the cusp tori inherit a Euclidean structure, we again consider how the links of equatorial vertices patch together. As in Figure \ref{cuspbigon},
we see how the rhombi corresponding to each of the equatorial vertices glue together to create the cusps.

\begin{figure}[htbp]
\centering
\includegraphics[scale=0.6]{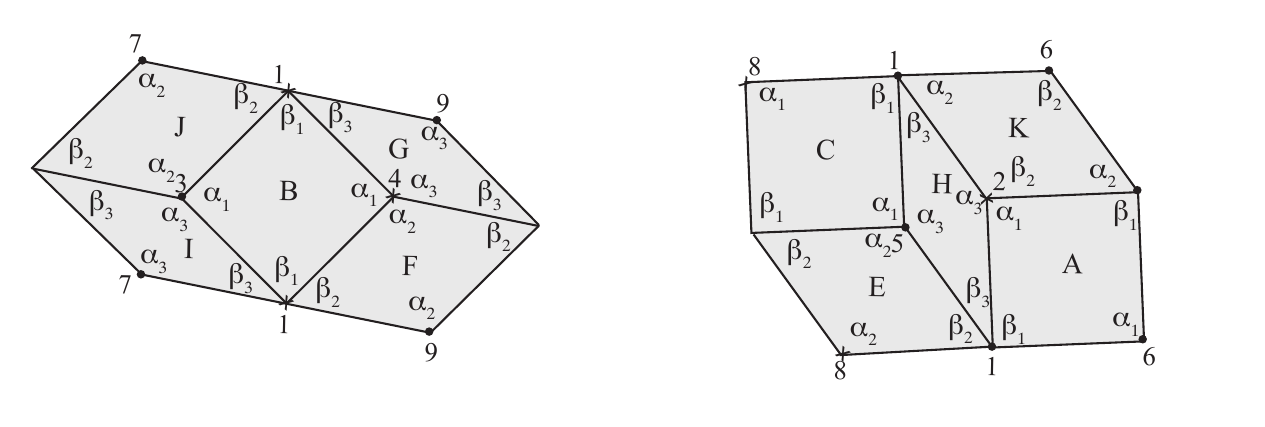}
\caption{Rhombi from each ideal vertex gluing together to create the cusp diagram near a bigon.}
\label{cuspbigon}
\end{figure}

Again, we see that around each vertex of the cusp tiling the rhombi glue with total angle $2\pi$ and that the top and bottom edges of the resulting figures glue to each other via translations, hence the holonomy of a meridian is a translation. Together, these imply as above that the holonomy of the longitude must also be a translation and we obtain a complete structure on the cusp $C$.

\para{Vertex gluings for $\mathbb{S}^2$ tilings}
We now consider the finite apexes in the case of $\mathbb{X} =\mathbb{S}^2$.  The face--pairing of the bipyramids immediately gives a hyperbolic structure to the space
\[ \bigcup B_{n_i}(\alpha_i) \setminus \{ \text{1--skeleton} \}\]
because there are natural hyperbolic charts around every point. In order to extend the hyperbolic structure to the entire space, we need only check that we can indeed glue the pieces of the hyperbolic charts around the edges and the finite vertices.

As we have seen above, the sum of the dihedral angles around each edge class is exactly $2 \pi$, and there is no shearing about an edge because we have always glued the faces by a unique isometry (more explicitly, observe that in any edge class the endpoints of the altitudes from the opposite vertex to the edge are always glued together, hence there is no shearing).

Around each finite vertex (that is, the two vertices coming from the apexes of the bipyramids), we have seen above that the link polygons fit together to give a tiling of the (round) sphere. Therefore the hyperbolic structures at the apexes fit together to give hyperbolic charts around the resulting finite vertices, and so we see that the hyperbolic structure on $\bigcup B_{n_i}(\alpha_i) \setminus \{ \text{1--skeleton} \}$ extends to $\bigcup B_{n_i}(\alpha_i) = \Ext(L)$.

\para{Boundary cusps for $\mathbb{E}^2$ tilings}
If $\mathbb{X} = \mathbb{E}^2$,  we have shown above that the links of the equatorial vertices induce Euclidean structures on each cusp coming from a strand of the link $L$. However, now that $\mathbb{X} = \mathbb{E}^2$ the link lives inside of a thickened torus,so we must show that the similarity structures on the cusps $T \times (0, \varepsilon)$ and $T \times (1- \varepsilon, 1)$ are indeed Euclidean structures. This in turn follows from our choice of the angles on the vertical edges of the bipyramids: the links of the vertices are regular polygons, all with the same edge length, that together tile the corresponding horosphere.

\para{Geodesic boundary for $\mathbb{H}^2$ tilings}
In the case when $\mathbb{X} = \mathbb{H}^2$, we actually take two copies of each $B_{n_i}(\alpha_i)$ and glue the copies together along their two truncation faces to induce a complete hyperbolic structure on the doubled link complement $D \Ext(L)$. By Theorem \ref{uniquetotgeostructure},
the two copies of $S$ corresponding to $S \times \partial I$ are totally geodesic in $D\Ext(L)$, so by cutting along these surfaces we arrive at a complete hyperbolic structure on $\Ext(L)$ whose boundary components are totally geodesic.
\end{proof}

Theorem \ref{computingvolumes} allows us to compute the volumes of alternating $k$-uniform tiling links by knowing the volumes of the corresponding symmetric $n$-bipyramids. For $g > 1$, we can use the  list of volumes of symmetric truncated $n$-bipyramids that appears in Figure \ref{bipyramidvolumetable}. For $g=1$,  there is a list of volumes of symmetric ideal $n$-bipyramids in \cite{Adams2015}, Table 1.

For $g =0$, as in the proof of Theorem \ref{computingvolumes}, splitting the bipyramidal decomposition of a spherical alternating $k$-uniform tiling link into top and bottom pyramids along their equatorial planes and gluing the tops and bottoms around the vertex links of $U$ and $D$ respectively, we recover Menasco's decomposition into two ideal polyhedra. Thus for an alternating link $L$ derived from a spherical $k$-uniform tiling we obtain a decomposition of $S^3 \setminus L$ into two copies of the polyhedron associated to the spherical tiling.

\begin{coro}
A alternating $k$-uniform tiling link $L \subset S^3$ corresponding to a spherical tiling has volume exactly twice that of the ideal hyperbolic polyhedron of the same combinatorial description with ideal vertices on the sphere at infinity placed according to the equilateral tiling on the sphere.
\end{coro}

Note that this fact is implicit in the work of  \cite{AR},  where explicit hyperbolic structures for links corresponding to Archimedean solids were determined. In \cite{Rivin}, it is proved that a maximally symmetric hyperbolic polyhedron has maximal volume for the polyhedra of the same combinatorial type, so these polyhedra are maximal volume of the combinatorial type.

The volumes of these polyhedra are enumerated in Figure \ref{sphericaltilingstable} along with approximate angles (to two decimal places) of their associated spherical tilings.

\begin{figure}[htbp]
\centering
\renewcommand{\arraystretch}{1.5}
\begin{tabular}{|| c | c | c | c ||}
\hline
Solid & Vertex configuration & Angles & $\text{vol}(L) / 2$ \\
\hline
\hline
tetrahedron & 3.3.3 & $\frac{2\pi}{3} . \frac{2\pi}{3} . \frac{2\pi}{3}$ & $v_\text{tet} \approx 1.0149$ \\
\hline
octahedron & 3.3.3.3 & $\frac{\pi}{2} . \frac{\pi}{2} . \frac{\pi}{2} . \frac{\pi}{2}$ & $v_\text{oct} \approx 3.6639$ \\
\hline
cube & 4.4.4 & $\frac{2\pi}{3} . \frac{2\pi}{3} . \frac{2\pi}{3}$ & 5.0747 \\
\hline
dodecahedron & 5.5.5 & $\frac{2\pi}{3} . \frac{2\pi}{3} . \frac{2\pi}{3}$ & 20.5802 \\
\hline
truncated tetrahedron & 3.6.6 & (1.17).(2.56).(2.56) & 8.2957 \\
\hline
cuboctahedron & 3.4.3.4 & (1.23).(1.91).(1.23).(1.91) & 12.0461 \\
\hline
truncated cube & 3.8.8 & (1.10).(2.59).(2.59) & 20.8916 \\
\hline
truncated octahedron & 4.6.6 & (1.68).(2.30).(2.30) & 25.2238 \\
\hline
rhombicuboctahedron & 3.4.4.4 & (1.13).(1.72).(1.72).(1.72) & 31.6987 \\
\hline
truncated cuboctahedron & 4.6.8 & (1.62).(2.18).(2.48) & 57.2688 \\
\hline
icosidodecahedron & 3.5.3.5 & (1.11).(2.03).(1.11).(2.03) & 39.8793 \\
\hline
truncated dodecahedron & 3.10.10 & (1.06).(2.61).(2.61) & 61.5356 \\
\hline
truncated icosahedron & 5.6.6 & (1.94).(2.17).(2.17) & 77.7139 \\
\hline
rhombicosadodecahedron & 3.4.5.4 & (1.08).(1.62).(1.96).(1.62) & 92.7191 \\
\hline
truncated icosidodecahedron & 4.6.10 & (1.59).(2.13).(2.57) & 155.4566 \\
\hline
\end{tabular}
\caption{Volumes of maximal ideal Archimedean solids.}
\label{sphericaltilingstable}
\end{figure}

\section{Volume density in thickened surfaces}
\label{sec:density}

In addition to the exact computations appearing above, bipyramidal decompositions of tiling link complements can also be used to demonstrate independence of the crossing number and the volume of a tiling link. Recall that the {\em volume density} of a tiling link $L$ is $\mathcal{D}(L) = \text{vol}(L) / c(L)$, where $c(L)$ is the crossing number of $L$.
See Figure \ref{tilingstable} for some example volume density computations.

After developing some preliminary bounds on volume density (Lemmas \ref{lem:vol_dens_g1} and \ref{lem:vol_dens_g2}), we show in Theorem \ref{2voctdensity} that the set of all volume densities is dense in the interval $[0, 2v_{\text{oct}}]$. We conclude the section with some open questions about the finer structure of the set of volume densities.

\begin{figure}[ht]
\centering
\renewcommand{\arraystretch}{1.5}
\begin{subfigure}{.49 \textwidth}
\centering
\begin{tabular}{|| c | c | c  ||}
\hline
Tiling link  & $\mathcal{D}(L)$\\
\hline
\hline
\begin{tabular}{@{}c@{}}\includegraphics[scale=.5]{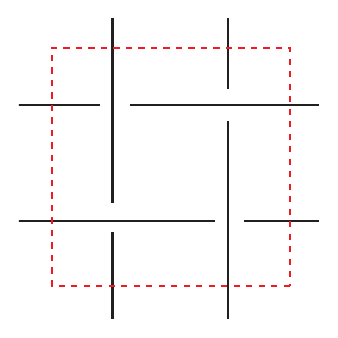} \\ 4.4.4.4 \end{tabular} & $v_\text{oct} \approx 3.6639$\\
\hline
\begin{tabular}{@{}c@{}}\includegraphics[scale=.5]{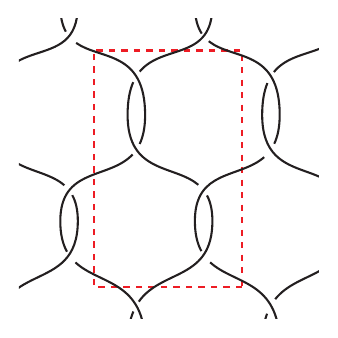} \\ 6.6.6 \end{tabular} & $3.0448$\\
\hline
\begin{tabular}{@{}c@{}}\includegraphics[scale=.5]{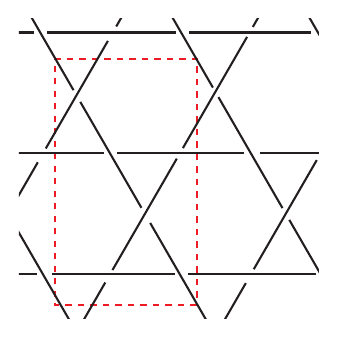} \\ 3.6.3.6 \end{tabular} & $3.3831$\\
\hline
\begin{tabular}{@{}c@{}}\includegraphics[scale=.5]{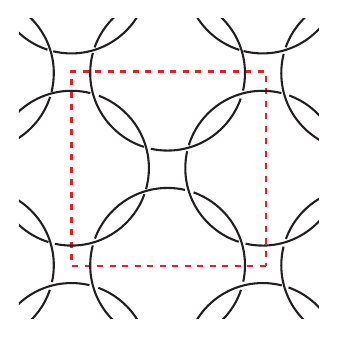} \\ 4.8.8 \end{tabular} & $2.8797$\\
\hline
\begin{tabular}{@{}c@{}}\includegraphics[scale=.5]{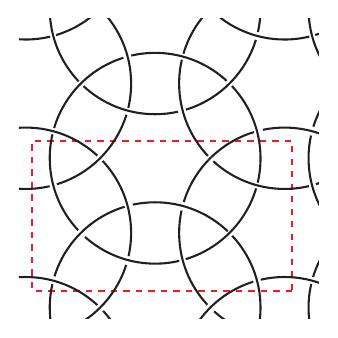} \\ 3.4.6.4 \end{tabular} & $3.5235$\\
\hline
\end{tabular}
\caption{Volume densities for Euclidean tiling links.}
\label{euclideantable}
\end{subfigure}
\begin{subfigure} {.5 \textwidth}
\centering
\begin{tabular}{|| c | c | c  ||}
\hline
Tiling link & Minimal genus & $\mathcal{D}(L)$\\
\hline
\hline
\begin{tabular}{@{}c@{}}\includegraphics[scale=.1]{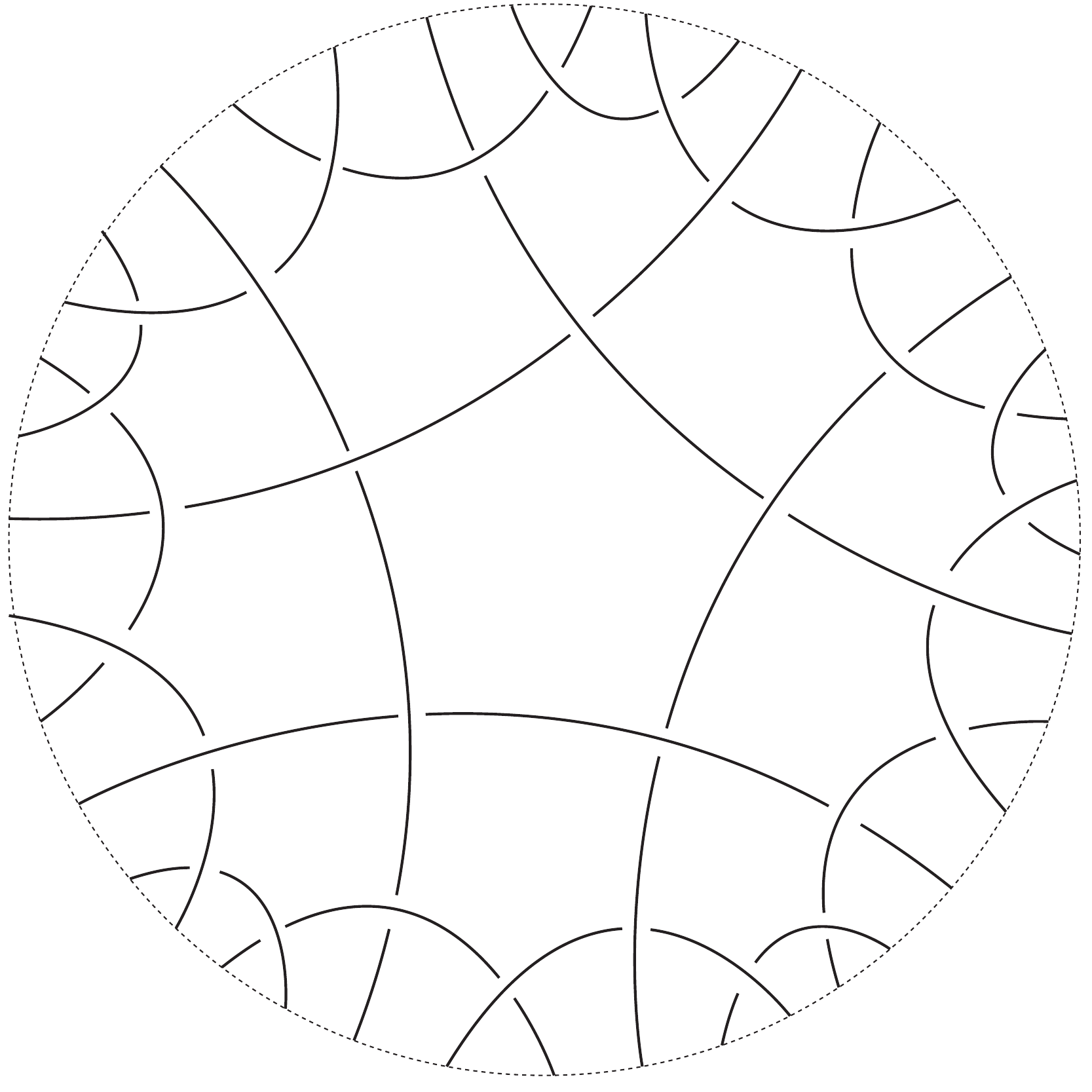} \\ 5.5.5.5 \end{tabular} & 2 & $5.4535$\\
\hline
\begin{tabular}{@{}c@{}}\includegraphics[scale=.1]{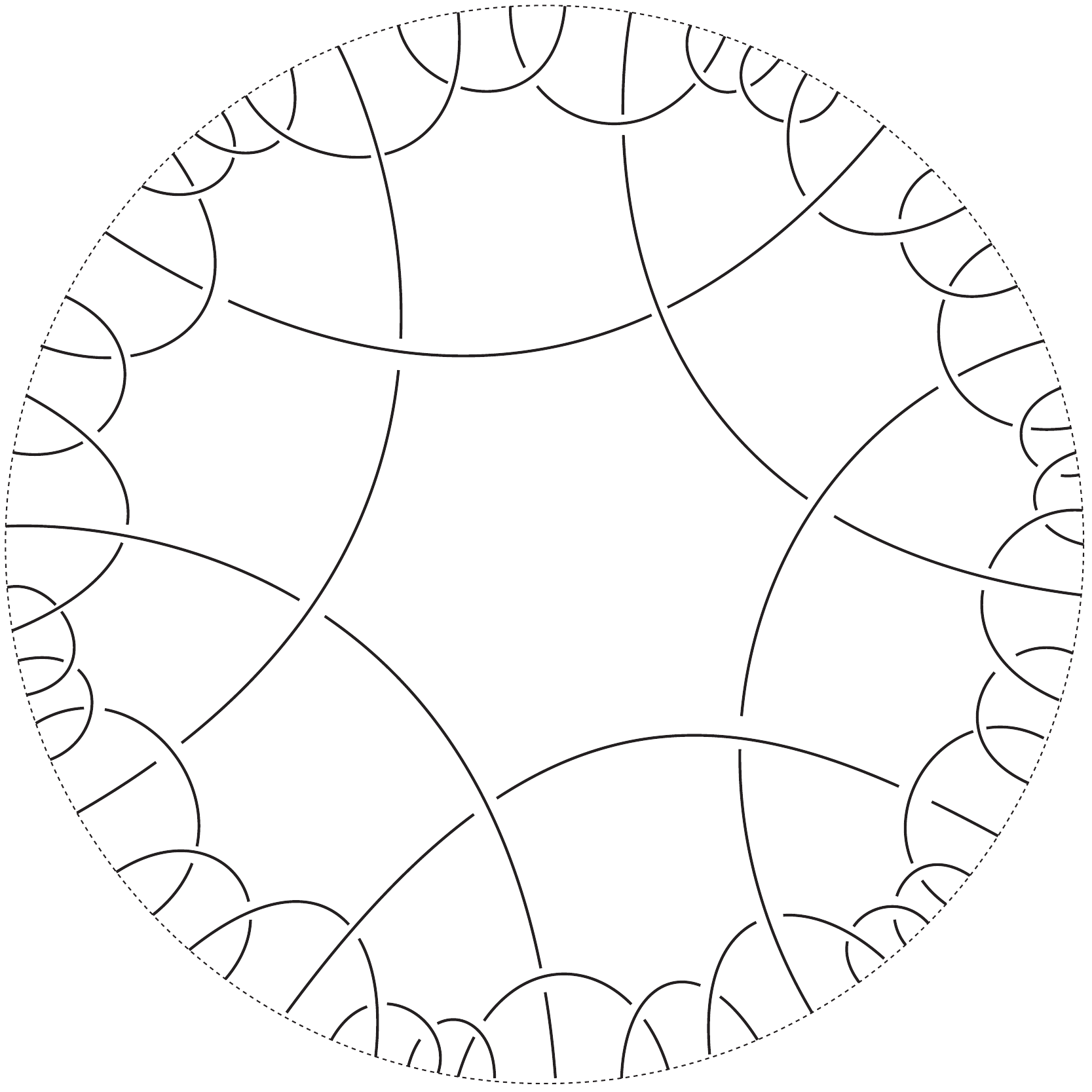} \\ 6.6.6.6 \end{tabular}  & 2 & $6.1064$\\
\hline
\begin{tabular}{@{}c@{}}\includegraphics[scale=.1]{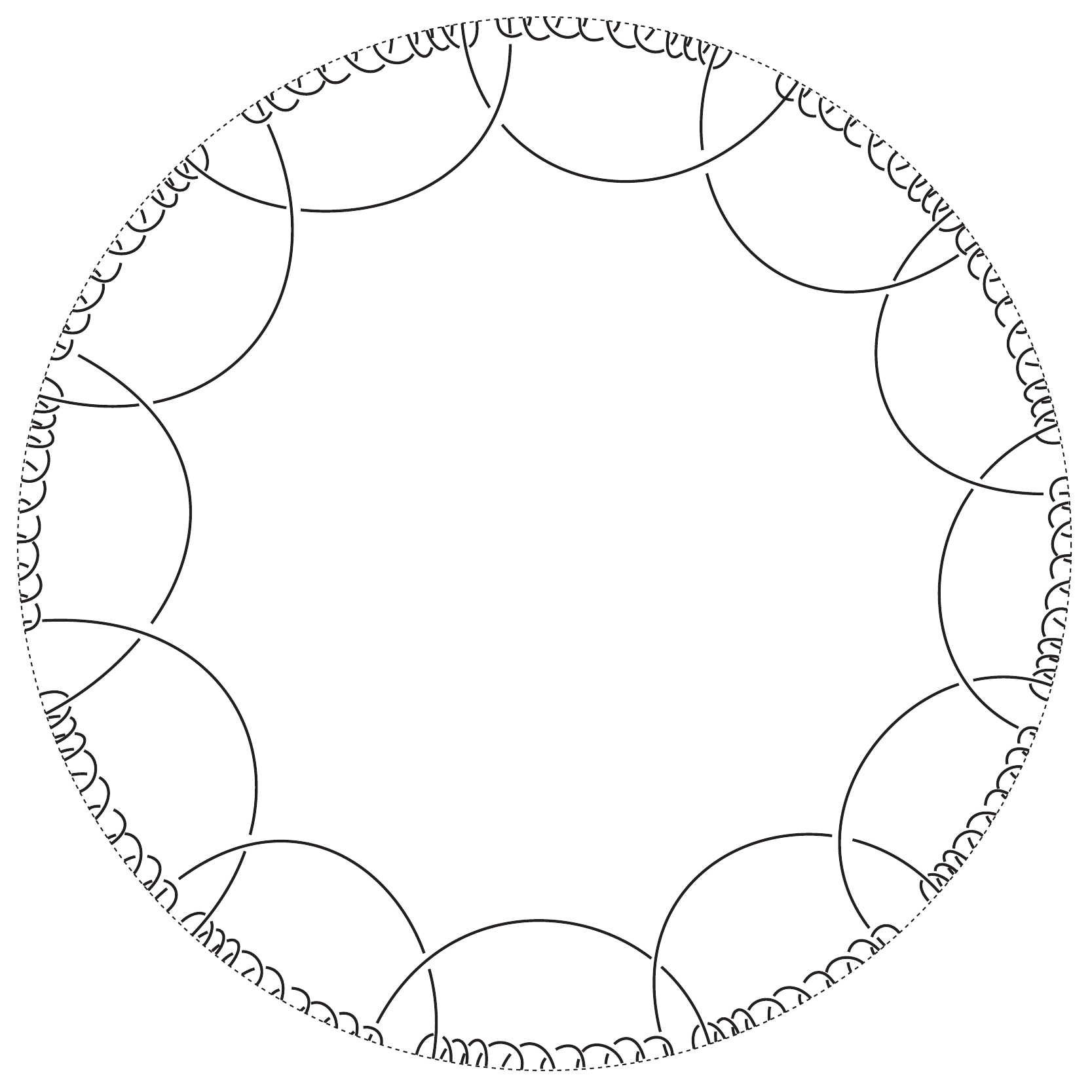} \\ 12.12.12.12 \end{tabular} & 2 & $7.0470$\\
\hline
\begin{tabular}{@{}c@{}}\includegraphics[scale=.1]{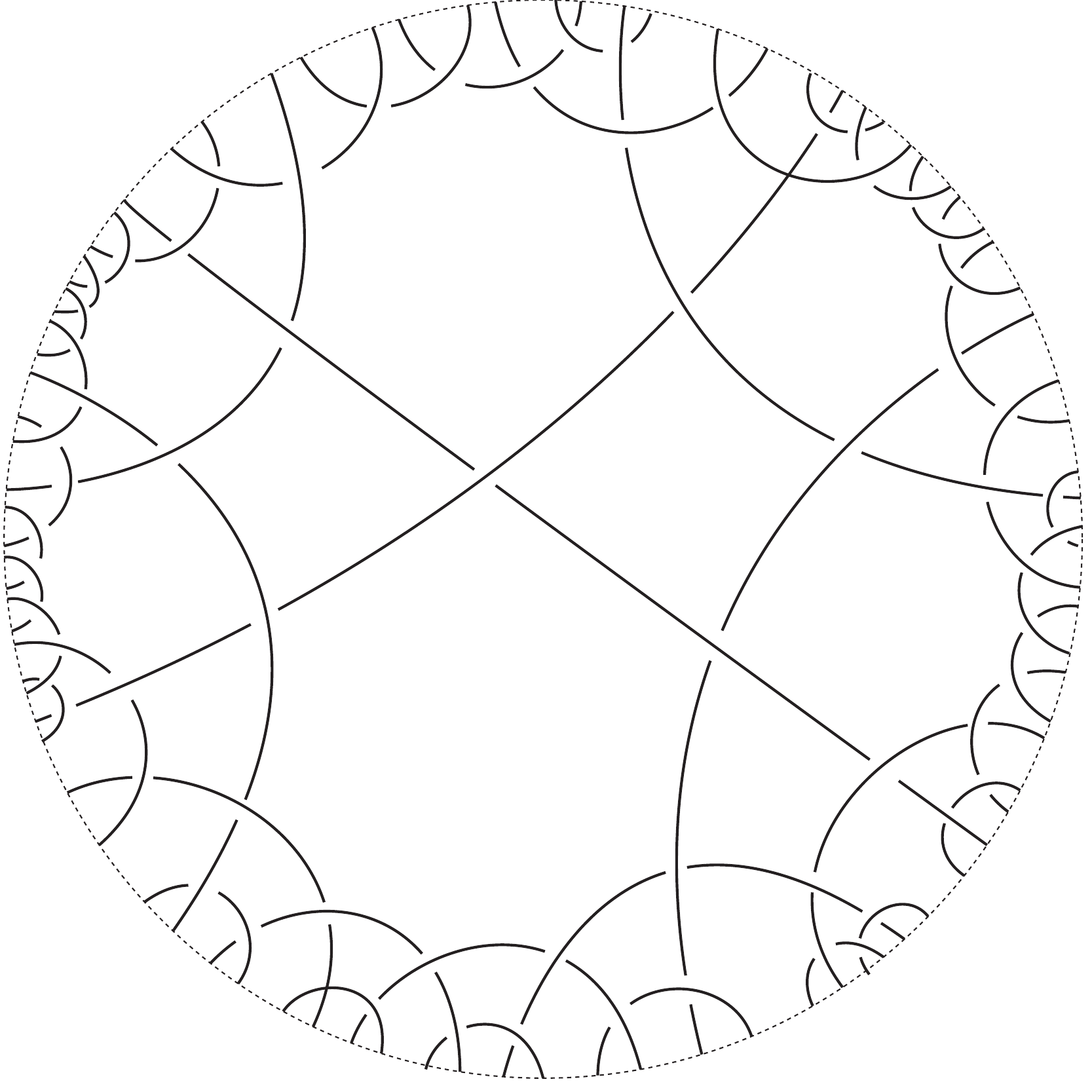} \\ 4.8.4.8 \end{tabular} & 2 & $5.4581$\\
\hline
\begin{tabular}{@{}c@{}}\includegraphics[scale=.1]{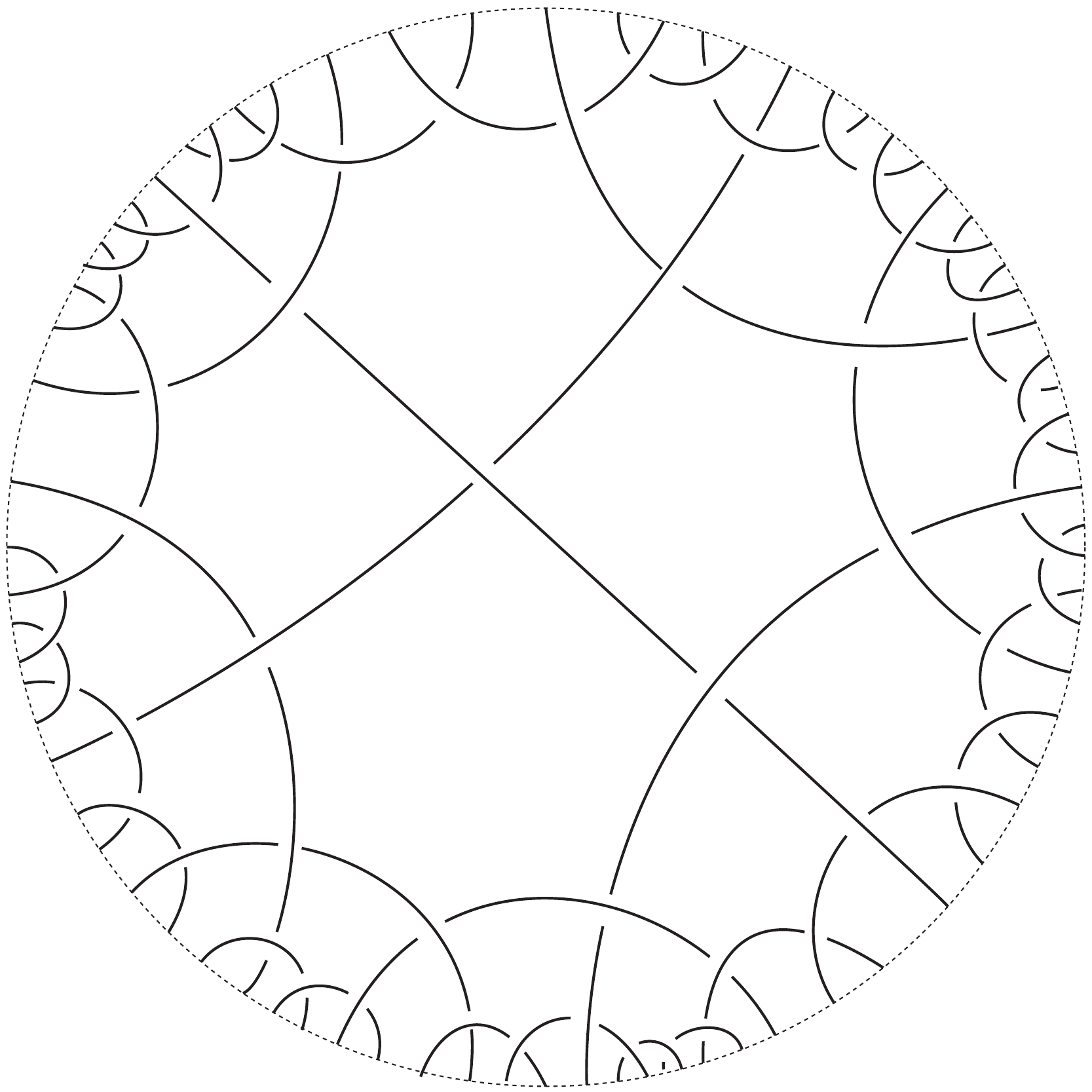} \\ 5.6.5.6 \end{tabular} & 3 & $5.7962$\\
\hline
\end{tabular}
\caption{Volume densities for hyperbolic tiling links.}
\label{hyptable}
\end{subfigure}
\caption{Volume densities for links in thickened surfaces.}
\label{tilingstable}
\end{figure}

We begin by observing that the generalized octahedral construction (Lemma \ref{lem:octdecomp}) with one octahedron per crossing and $U$ and $D$ ideal immediately yields the following.


\begin{lem}\label{lem:vol_dens_g1}
The volume density of a link $L$ in a thickened torus $T \times (0,1)$ is at most $v_\text{oct}$.
\end{lem}

The upper bound is realized by a link in $T \times (0,1)$ by taking the quotient of the infinite square weave link via a $\mathbb{Z} \times \mathbb{Z}$ subgroup of the symmetry group of the square tiling.

Corollary \ref{cor:octvolbd} and Theorem \ref{computingvolumes} now allow us to give equivalent results for $S \times I$.

\begin{lem}\label{lem:vol_dens_g2}
The volume density of a link $L$ in a thickened surface $S \times I$ is  strictly less than $2v_\text{oct}$.
\end{lem}

\begin{lem}\label{maxSxIseq}
There exists a sequence of links $L_n$ in thickened surfaces $S_n \times I$ with volume densities approaching $2v_\text{oct}$.
\end{lem}

\begin{proof}
For $n>5$, let $\widetilde{L_n}$ be an infinite alternating $k$-uniform tiling link derived from the $n.n.n.n$ tiling of $\mathbb{H}^2$ and $S_n$ be a closed surface such that its fundamental group is a subgroup of the symmetries of $\widetilde{L_n}$. Suppose a fundamental domain for $S_n$ is composed of $j$ $n$-gons, and let $L_n$ denote the tiling link in $S_n \times I$ after quotienting by the appropriate surface group. Note that there will be $j$ $n$-gons in a projection of $L_n$ onto $S_n \times \{0\}$. By Theorem \ref{computingvolumes} the volume of $(S_n \times I) \setminus L_n$ is $j\text{vol}(B_n^\square)$ where $B_n^\square$ denotes the generalized $n$-bipyramid with all dihedral angles equal to $\pi / 2$.
The crossing number of the resulting link $L_n$ in $S_n \times I$ is seen to be $nj/4$ by results in \cite{AFLT02} and \cite{Fleming}. Hence the volume density of $L_n$ in $S_n \times I$ is $4\text{vol}(B_n^\square)/{n}$. Observe that $B_n^\square$ breaks up into $n$ isometric generalized tetahedra with dihedral angles $A = 2\pi/n, D = \pi/2$ and $B=C=E=F=\pi/4$. Sending $n$ to infinity makes these wedges tend to the maximal tetrahedron from Corollary \ref{maximalBipyramid}, so just as with $B_n^{trunc}$, we see as in the discussion following Corollary \ref{maximalBipyramid},
\[\lim_{n \rightarrow \infty} \frac{4\text{vol}(B_n^\square)}{n} = 2v_\text{oct}.\]
\end{proof}

\begin{thm}\label{2voctdensity}
The set of volume densities of links in $S_g \times I$ across all genera $g \ge 2$, assuming totally geodesic boundary, is a dense subset of the interval $[0, 2v_\text{oct}]$.
\end{thm}

\begin{proof}
Let $x \in (0, 2v_\text{oct})$ and $\epsilon > 0.$ We will find a link $L$ such that its density is within $\epsilon$ of $x$.  By Lemma \ref{maxSxIseq} and its proof, there exists an alternating link $L$ in a thickened surface $S \times I$ coming from a $q.q.q.q$ tiling of $\mathbb{H}^2$ with $\text{vol}(L ) = k\text{vol}(B_q^\square)$, such that $\mathcal{D}(L) = \text{vol}(L)/c(L) > x$. 
By taking an $n$--fold cyclic cover of $(S \times I) \setminus L$, we obtain a link $L'$ in $S' \times I$, where $S'$ has higher genus than $S$, such that both the volume and crossing number of $L$ have been scaled up by  a factor of $n$.

As in Figure \ref{augment}(a) and (b), we choose two strands in the projection on the same complementary region and we add a trivial component $J$ so that the resulting projection of $L'' = L' \cup J$ is still alternating. By Theorem \ref{tg}, the complement of $L''$ in $S' \times I$ is hyperbolic. 

Drilling out $J$  raises the volume of $L'$ by some amount $b_n > 0$ which may depend on $n$ and the particular choice of drilling. 
We rearrange the projection as in Figure \ref{augment}(c). Then we perform $(1,t_n)$-Dehn filling on $J$, replacing the two original strands with a  region of $2 t_n+1$ crossings to create a link $L'''$, as for example appears in Figure \ref{augment}(d). The resulting link $L'''$ is alternating, and hence by Theorem \ref{tg} is also hyperbolic.

The volume of the filled link $L'''$ is at most $n\text{vol}(L) + b_n$, and the crossing number of $L'''$ is $nc(L) + 2t_n+ 1$ since a reduced alternating projection in $S \times I$ realizes the minimal crossing number (see \cite{AFLT02}). 

\begin{figure}[htbp]
\centering
\includegraphics[scale=0.8]{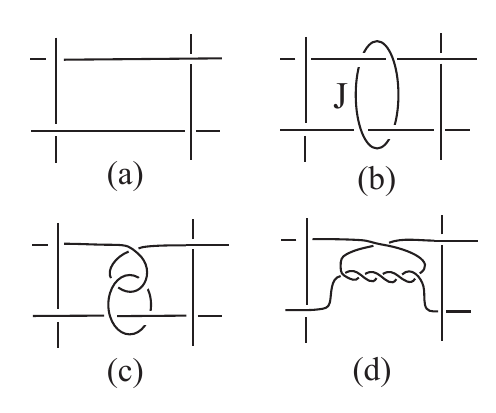}
\caption{Augmenting an alternating link in $S \times I$.}
\label{augment}
 \end{figure}

To show that the density of the resulting link $L'''$ is less than $x+\epsilon$, we set
\[t_n = \left\lfloor \frac{n\text{vol}(L) + b_n}{2x} \right\rfloor - \frac{nc(L)}{2}\]
for each $n$. From the discussion above, we can calculate that
\[\mathcal{D}(L''') \le \frac{n\text{vol}(L)+b_n}{nc(L)+2t_n + 1} 
= x+ \frac{n\text{vol}(L) + b_n - x \left( 2\left\lfloor \frac{n\text{vol}(L) + b_n}{2x} \right\rfloor +1 \right) }{2\left\lfloor \frac{n\text{vol}(L) + b_n}{2x} \right\rfloor +1}\]
and by increasing the numerator by $x$ and decreasing the denominator by $1$ on the right hand side we see
\[\mathcal{D}(L''') < x+ \frac{n\text{vol}(L) + b_n - 2x \left\lfloor \frac{n\text{vol}(L) + b_n}{2x} \right\rfloor }{2\left\lfloor \frac{n\text{vol}(L) + b_n}{2x} \right\rfloor} < x+ \frac{2x}{{2\left\lfloor \frac{n\text{vol}(L) + b_n}{2x} \right\rfloor} } \]
where the last inequality follows from the fact that for any $A, B >0$, 
\[A-B\left\lfloor \frac{A}{B} \right\rfloor = B\left( \frac{A}{B}-\left\lfloor \frac{A}{B} \right\rfloor \right) <B.\]
Now we can make the denominator on the right hand side arbitrarily large and so by setting $n$ large enough we get
\[\mathcal{D}(L''') < x+ \epsilon.\]

To prove that the volume density does not drop too much throughout this procedure, we employ a result of Futer, Kalfagianni, and Purcell \cite{Futer08}, which provides a lower bound for volume change under Dehn filling along a slope $s$:
\[\text{vol}(L''') \ge \left( 1 - \left( \frac{2\pi}{\ell} \right)^2 \right) ^{3/2} \text{vol}(L'')\]
where $\ell$ is the length of $s$. Since $|t_n|$ (and thus $\ell$) increases without bound as $n$ does, we can choose $n$ sufficiently large such that
\[\left( 1 - \left( \frac{2\pi}{\ell} \right)^2 \right) ^{3/2} > \frac{x}{x + \epsilon}.\]
Now observe that by our choice of $t_n$ we have
\[x(nc(L)+2t_n+1) = x \left( 2 \left\lfloor \frac{n\text{vol}(L)+b_n}{2x} \right\rfloor +1 \right) \le (n\text{vol}(L) + b_n) + x = \text{vol}(L'') +x \]
and hence we get that
\[\frac{\text{vol}(L'')}{nc(L)+2t_n+1} \ge x - \frac{x}{nc(L)+2t_n+1} > x - \frac{\epsilon}{x}\]
for large enough $n$. Therefore putting all of our estimates together, we see that for sufficiently large $n$,
\[\mathcal{D}(L''')
\ge \left( 1 - \left( \frac{2\pi}{\ell} \right)^2 \right) ^{3/2} \frac{\text{vol}(L'')}{nc(L)+2t_n+1}
\ge \left(\frac{x}{x + \epsilon} \right) \left( x - \frac{\epsilon}{x} \right) 
= \frac{x^2-\epsilon}{x+\epsilon}= x-\epsilon.\]
\end{proof}

Note that cyclic covers of tori remain tori and hence the following result can be deduced through arguments similar to those above.

\begin{coro}\label{TxIdensity}
The set of volume densities for links in $T \times (0,1)$ is a dense subset of the interval $[0, v_{oct}]$.
\end{coro}

Taking cyclic covers of a surface of genus $g \ge 2$ increases the number of handles, and so it is natural to ask about the set of volume densities in a surface of a fixed genus. We define the genus of a non-orientable surface in terms of the Euler characteristic $\chi$ as $g = (2 - \chi)/2$, so that ``fixed genus'' directly translates to ``fixed Euler characteristic.''

Suppose we fix a surface $S_g$ of genus $g$. Then the maximal volume density of a link in $S_g \times I$ is achieved when the average volume per crossing is highest. As the average volume per crossing increases as we increase the number of edges in our bipyramid, we are trying to tile $S_g$ using as few polygons as possible with the largest number of sides. By Euler characteristic considerations, the minimum number of crossings comes from a tiling by a single $(8g-4)$-gon.

From that logic, we propose the following conjecture:

\begin{conj}
The maximum volume density for links in thickened surfaces of fixed genus $g \ge 2$ is given by
\[\beta_g:= \frac{\emph{vol}(B_{8g-4}^\square)}{2g-1}\]
and is realized by an alternating $k$-uniform tiling link derived from a tiling by right-angled $(8g-4)$-gons in a closed non-orientable surface of genus $g$.
\end{conj}

As the proof of Theorem \ref{2voctdensity} involves increases in  genus, it cannot be employed to show density of volume densities in $[0, \beta_g]$. However, augmentation and taking a belted sum with links in $S^3$ leaves genus intact, and thus the set of volume densities of links in $S_g \times I$ can be shown to be dense over at least the interval $[0, v_{oct}]$ by employing results of, e.g., \cite{Burton2015} and \cite{Adams2017}.




We conclude with a few open questions.

\begin{ques}
Is the set of volume densities of links in $S_g \times I$ for fixed genus $g \geq 2$ dense in $[0,\beta_g]$?
\end{ques}

\begin{ques}
Can bipyramidal decompositions be used to compute the volume of non-alternating tiling links?
\end{ques}

\begin{ques}
Can bipyramidal decompositions be used to compute the volumes of links derived from non-$k$-uniform tilings?
\end{ques}

In both of Questions 2 and 3, the bipyramids will not remain regular, and some skewing factor should be taken into account to determine the complete hyperbolic structure and calculate exact volumes. But bipyramids can substantially cut down the number of equations produced as compared to when one tries to find the hyperbolic structure by cutting into ideal tetrahedra and satisfying the gluing and completeness equations.

\
\bibliographystyle{amsalpha}
\bibliography{SMALL}

\end{document}